\documentclass[10pt, english]{amsart}

\usepackage{amsmath,amssymb,enumerate, amscd}

\usepackage{xcolor}
\usepackage[T1]{fontenc}
\usepackage[all]{xy}

\usepackage{babel}
\usepackage{amstext}
\usepackage{amsmath}
\usepackage{amsfonts}
\usepackage{latexsym}
\usepackage{ifthen}

\usepackage{xypic}
\xyoption{all}
\newcommand{\id}{{\rm id}}

\newcommand{\rk}{{\rm rk}}

\newtheorem{lemma1}{}[section]

\newenvironment{lemma}{\begin{lemma1}{\bf Lemma.}}{\end{lemma1}}
\newenvironment{example}{\begin{lemma1}{\bf Example.}\rm}{\end{lemma1}}

\newenvironment{theorem}{\begin{lemma1}{\bf Theorem.}}{\end{lemma1}}
\newenvironment{proposition}{\begin{lemma1}{\bf Proposition.}}{\end{lemma1}}
\newenvironment{corollary}{\begin{lemma1}{\bf Corollary.}}{\end{lemma1}}
\newenvironment{remark}{\begin{lemma1}{\bf Remark.}\rm}{\end{lemma1}}

\newenvironment{definition}{\begin{lemma1}{\bf Definition.}}{\end{lemma1}}

\newenvironment{conjecture}{\begin {lemma1}{\bf Conjecture.}}{\end{lemma1}}
\newenvironment{question}{\begin{lemma1}{\bf Question.}}{\end{lemma1}}

\newenvironment{remark*}{{\bf Remark.}}{}
\newenvironment{example*}{{\bf Example.}}{}
\newenvironment{assumption*}{{\bf Assumption.}}{}

\newcommand{\R}{\ensuremath{\mathbb{R}}}
\newcommand{\Q}{\ensuremath{\mathbb{Q}}}
\newcommand{\Z}{\ensuremath{\mathbb{Z}}}
\newcommand{\C}{\ensuremath{\mathbb{C}}}
\newcommand{\N}{\ensuremath{\mathbb{N}}}
\newcommand{\PP}{\ensuremath{\mathbb{P}}}

\usepackage{eurosym}

\newcommand{\merom}[3]{\ensuremath{#1:#2 \dashrightarrow #3}}

\newcommand{\holom}[3]{\ensuremath{#1:#2  \rightarrow #3}}
\newcommand{\fibre}[2]{\ensuremath{#1^{-1} (#2)}}

\makeatletter
\ifnum\@ptsize=0  \fi \ifnum\@ptsize=2  \fi 

\newcommand\sE{{\mathcal E}}

\newcommand\sF{{\mathcal F}}
\newcommand\sG{{\mathcal G}}

\newcommand\sI{{\mathcal I}}

\newcommand\sO{{\mathcal O}}

\newcommand\sV{{\mathcal V}}

\newcommand\ep{{\epsilon}}

\newcommand\sW{{\mathcal W}}

\newcommand{\Chow}[1]{\ensuremath{\mbox{\rm Chow}(#1)}}

\title{Direct images of pseudoeffective cotangent bundles} 
\date{February 27, 2023}

\author{Junyan Cao}
\author{Andreas H\"oring}

\address{Junyan Cao, Universit\'e C\^ote d'Azur, CNRS, LJAD, France}
\email{junyan.cao@unice.fr}

\address{Andreas H\"oring, Universit\'e C\^ote d'Azur, CNRS, LJAD, France}
\email{Andreas.Hoering@unice.fr}


\begin{document}

\begin{abstract}  
Let $A$ be an elliptic curve, and let $V_A$ be the Serre vector bundle on $A$.
A famous example of Demailly-Peternell-Schneider shows that
the tautological class of $V_A$ contains a unique closed positive current.
In this survey we start by generalising this statement to arbitrary compact K\"ahler manifolds. We then give an application to abelian fibrations $X \rightarrow Y$ where the total space $X$ has pseudoeffective cotangent bundle and raise some questions about nonvanishing properties of these bundles.
\end{abstract}

\newpage

\maketitle

\section{Introduction}

\subsection{An example of Demailly-Peternell-Schneider}

The notion of singular metrics on line bundles, introduced by Jean-Pierre Demailly \cite{Dem92b, Dem92}, are indispensable to have analytic analogues of algebraic notions of positivity of line bundles as shown by the following celebrated example:

\begin{example} \cite[Example 1.7]{DPS94} \label{example-DPS}
Let $A$ be an elliptic curve, and let
$$
0 \rightarrow \sO_A \rightarrow V_A \rightarrow \Omega_A \simeq \sO_A \rightarrow 0
$$
be an unsplit extension (i.e. $V_A$ is the Serre vector bundle). Let $\zeta_A
:=c_1(\sO_{\PP(V_A)}(1))$ be the tautological class on the projectivised bundle
$\PP(V_A)$. Then $\zeta_A$ is nef and the unique positive current in the cohomology class of $\zeta_A$ is the current of integration along the curve $\PP(\Omega_A) \subset \PP(V_A)$.
In particular the nef line bundle $\sO_{\PP(V_A)}(\zeta_A)$ does not admit a smooth metric with semipositive curvature.
\end{example}

The assumption that the extension is unsplit is equivalent to assuming that $H^0(A, V_A) \simeq \C$. It is quite remarkable that this algebraic condition leads to such a strong restriction on transcendent elements in the class $\zeta_A$.
Our first result is a natural generalisation of this example
to compact K\"ahler manifolds of arbitrary dimension:

\begin{theorem}\label{theorem-dpsgen}
Let $X$ be a compact K\"ahler manifold, and let 
\begin{equation} \label{extension}
0 \rightarrow \sO_X \rightarrow V\rightarrow   \bigoplus_{i=1}^{r} \sO_X \rightarrow 0,
\end{equation}
be an extension of vector bundles such that $H^0(X, V) \simeq \C$.
Let $T$ be a positive closed current in the tautological class $c_1\big(\sO_{\PP(V)}(1)\big)$.
Then $T$ is the current of integration along the divisor $\PP(\bigoplus_{i=1}^{r} \sO_X ) \subset \PP(V)$.
\end{theorem}

While rather classical arguments allow to see that
$\kappa(\PP(V), \sO_{\PP(V)}(1))=0$, proving that the cohomology class contains no other positive closed currents requires direct image techniques developed only in the last ten years, cf. Proposition \ref{exvec}.
In fact we will prove a slightly more general statement describing currents in the tautological class whenever we consider an extension of the type  \eqref{extension}, cf. Theorem \ref{genDPS}. This description allows us to obtain a positivity result for the ``direct image'' 
of the cotangent bundle:

\begin{theorem}  \label{theorem-direct-image}
	Let $\holom{f}{X}{Y}$ be a fibration from a compact K\"ahler manifold $X$ onto
	smooth compact curve $Y$.
	Assume that the general fibre $F$ is a complex torus,  denote by $\sV \subset \Omega_X$ the saturation of $f^* f_* \Omega_X \rightarrow \Omega_X$.  If $\Omega_X$ is a pseudoeffective vector bundle, then $\sV$ is pseudoeffective.
\end{theorem}

This statement is surprising for two reasons: in general positivity properties are inherited by quotient sheaves, not subsheaves. Moreover it is not difficult 
to find examples where the direct image $f_* \Omega_X$ is an anti-ample vector bundle (cf. Example \ref{examplepushisotrivial}), so the positivity of $\sV$ only appears after saturating the image. Theorem \ref{theorem-direct-image} has no obvious
generalisation to fibres of arbitrary Kodaira dimension: if $F$ is projective manifold such that $\Omega_F$ is ample but $q(F)=0$ (e.g. $F \subset \PP^N$ is a complete intersection
of sufficiently high degree and codimension \cite{BD18, Xie18}), the product $X:= \PP^1 \times F$
has pseudoeffective cotangent bundle, but $\sV = p_{\PP^1}^* (p_{\PP^1})_* \Omega_X
\simeq  p_{\PP^1}^* \Omega_{\PP^1}$ is not pseudoeffective, where $p_{\PP^1}: X \to \PP^1$ is the natural projection.

Since Theorem \ref{theorem-direct-image} is essentially proven by restricting to a general fibre one is tempted to believe that the proof carries over without changes to a higher-dimensional base $Y$. However the proof of Theorem \ref{genDPS} relies  on the form of the extension \eqref{extension}, in fact Example \ref{example-bad-case} shows that the statement of Theorem \ref{genDPS} is already false for arbitrary extensions of the form 
$$
0 \rightarrow \sO_X^{\oplus 2} \rightarrow V  \rightarrow \sO_X^{\oplus 2}
 \rightarrow 0.
$$
In order to make further progress in this direction it seems necessary to have an argument that actually uses that the extension arises from the cotangent sequence of the fibration $f$.

\subsection{Nonvanishing conjecture for cotangent bundle}

In \cite{HP20} the second named author and Peternell introduced the following
conjecture:

\begin{conjecture} \label{conj:nv} 
Let $X$ be a projective manifold\footnote{In \cite{HP20} the conjecture is stated for varieties with klt singularities, in view of possible reduction steps via the MMP it is indeed essential to work in this generality. In this survey we work only on smooth spaces in order to preserve the simplicity of the exposition.}, and let $1 \leq q \leq \dim X$. Then $\Omega^{q}_X$ is pseudoeffective
if and only if
for some positive integer $m$ one has  
$$
H^0(X,S^{m}\Omega^{q}_{X}) \ne 0.
$$ 
\end{conjecture} 

This conjecture generalises the well-known nonvanishing conjecture for the canonical class, but is far from being obvious even for surfaces of general type. A general strategy
to attack this problem is to split it into two parts: varieties of (log-)general type
and varieties that admit fibrations with fibres of non-positive Kodaira dimension.
In this survey we single out a situation that is geometrically rather simple, but leads
to some interesting problems about the positivity of vector bundles:

\begin{question} \label{question-curve}
Let $X$ be a projective manifold such that the cotangent bundle $\Omega_X$ is pseudoeffective. Assume that $K_X$ is semiample with numerical dimension one, so the Iitaka fibration is a morphism $\holom{f}{X}{C}$ onto a curve $C$.

What can we say about the fibration? Do we have 
$$
H^0(X, S^k \Omega_X) \neq 0
$$ 
for some $k \in \N$?
\end{question}

If the curve $C$ has genus at least one, or more generally $q(X) \neq 0$,
this explains the pseudoeffectivity of $\Omega_X$ and nothing else can be said.
Thus we can assume that $q(X)=0$, in particular $C \simeq \PP^1$.
Applying the Beauville-Bogomolov decomposition to the general fibre (cf. Lemma \ref{lemmafactorisation}), we see that most difficult case 
is when the general fibre $F$ is an abelian variety.
In this case we can apply Theorem \ref{theorem-direct-image} to get

\begin{corollary} \label{corollary-rank-one}
Let $X$ be a projective manifold such that the cotangent bundle $\Omega_X$ is pseudoeffective. Assume that $X$ admits an abelian fibration
$\holom{f}{X}{C}$ onto a curve $C$.
If $f_* \Omega_X$ has rank one, there exists a finite \'etale cover $X' \rightarrow X$
such that $q(X') \neq 0$. In particular one has $H^0(X, S^k \Omega_X) \neq 0$ for some $k \in \N$.
\end{corollary}

For an elliptic fibration the condition that $f_* \Omega_X$ has rank one simply means that
$f$ is not isotrivial, so this generalises \cite[Thm.1.2a)]{HP20}.
The opposite case where $\rk f_* \Omega_X = \dim X$ corresponds to the situation
where $f$ is isotrivial: in this case we have a natural splitting
of $f_* \Omega_X$  (cf. Lemma \ref{lemma-splitting-image}), this reflects the property that $X$ is, after finite base change, birational to a product.
We expect that this ``product structure'' allows to show the nonvanishing by an explicit computation, in particular it should be possible to verify Conjecture \ref{conjecture-dichotomy} in this situation.

This leads us to the intermediate case where $1< \rk f_* \Omega_X < \dim X$
(see Subsection \ref{subsection-example} below for an example).
Since $\sV \subsetneq \Omega_X$, we can consider the distribution
$$
\sF := (\Omega_X/\sV)^* \subset T_X
$$
which is in fact a linear foliation over the smooth locus of the abelian fibration $f$. 

\begin{question}
Let $X$ be a projective manifold such that the cotangent bundle $\Omega_X$ is pseudoeffective. Assume that $X$ admits an abelian fibration
$\holom{f}{X}{C}$ onto a curve $C$. If $1< \rk f_* \Omega_X < \dim X$, is the foliation
$\sF \subset T_X$ algebraically integrable, i.e. is the general leaf an abelian variety?
\end{question}

It follows from the proof of Theorem \ref{theorem-direct-image} that the 
quotient $\Omega_X/\sV$ is contained in the singular locus of every positive current
in the tautological class $\zeta_X$. This seems to indicate that the cotangent sheaf $\sF^*$ of the foliation satisfies the condition of the criterion of Campana-P\v aun \cite{CP19, Dru18}, but e.g. for $\rk \sF=1$ this is never true by Brunella's theorem \cite{Bru06} (cf. also Example \ref{example-nef}). 
Therefore it would be very interesting to have an algebraic integrability criterion that takes into account not only the positivity of the foliation $\sF$, but also its embedding $\sF \subset T_X$.

If one wants to exploit Theorem \ref{theorem-direct-image} more thoroughly,
we can try to obtain stronger properties for the pseudoeffective subsheaf $\sV \subset \Omega_X$. 

\begin{conjecture} \label{conjecture-dichotomy}
Let $X$ be a projective manifold such that the cotangent bundle $\Omega_X$ is pseudoeffective. Assume that $X$ admits an abelian fibration
$\holom{f}{X}{C}$ onto a curve $C$. Let $F$ be a general fibre of $f$.
Then $\sV \subset \Omega_X$ contains a non-zero  reflexive subsheaf $\sW$ such that one of the following holds:
\begin{itemize}
\item $\sW$ is numerically projectively flat (cf. Definition \ref{definition-projectively-flat}); 
\item for some $\varepsilon>0$ the $\Q$-twist $\sW \otimes\sO_X(-\varepsilon F)$ is pseudoeffective.
\end{itemize}
\end{conjecture}

Note that we do not claim that $\sW \subset \Omega_X$ is saturated, in fact it might be indispensable not to saturate to avoid some negative quotient bundles.
If Conjecture \ref{conjecture-dichotomy} holds, we immediately obtain a positive answer to the nonvanishing in Question \ref{question-curve}, see Propositions \ref{proposition-projectively-flat}
and \ref{proposition-base-positivity}.
As a first step towards the conjecture we show, under some conditions, in Proposition \ref{proposition-dichotomy} that  $\kappa(X, \sV) \geq 0$ unless $\sV$ is nef in codimension one. While this gives us some confidence in the conjecture, it does not allow to conclude even if $\rk \sV=2$.

\begin{question}
Let $X$ be a projective manifold, and let $\sV$ be a reflexive sheaf of rank two on $X$.
Assume that $\sV$ is almost nef \cite{DPS01}, \cite[Defn.1.7]{HP19} and the determinant $\det \sV$ is nef but not big. Is the vector bundle $\sV$ nef ?
\end{question} 

If $\det \sV$ is numerically trivial, the answer is yes \cite[Thm.1.8]{HP19}. The question also has a positive answer if $\sV$
is the cotangent bundle of an elliptic surface \cite[Sect.6.D]{HP20}, 
but the argument relies heavily on Kodaira's classification of singular fibres \cite{Kod}.

\subsection{An example} \label{subsection-example}

Theorem \ref{theorem-direct-image} establishes the pseudoeffectivity for the saturation
of the direct image $f^* f_* \Omega_X$, while $f_* \Omega_X$ is typically very negative. Even in geometrically restrictive situations this makes it difficult
to understand the positivity of (the subsheaves of) the cotangent bundle:

\begin{example} \label{examplefibreproduct}
Let $\holom{f_S}{S}{\PP^1}$ be a relatively minimal elliptic fibration that is not isotrivial, and let $\holom{f_Z}{Z}{\PP^1}$ be a relatively minimal elliptic fibration that is isotrivial. Assume that $f_S$
and $f_Z$ do not have a singular fibre over the same point of $\PP^1$, so the fibre product $X: = S \times_{\PP^1} Z$ is smooth and admits two elliptic fibrations
$\holom{g_S}{X}{S}$ and $\holom{g_Z}{X}{Z}$. Then we have two exact sequences
$$
0 \rightarrow g_Z^* \Omega_Z \rightarrow \Omega_X \rightarrow \Omega_{X/Z} \simeq g_S^* \Omega_{S/\PP^1} \rightarrow 0
$$
and
$$
0 \rightarrow g_S^* \Omega_S \rightarrow \Omega_X \rightarrow \Omega_{X/S} \simeq g_Z^* \Omega_{Z/\PP^1} \rightarrow 0.
$$
Denote by 
$$
\holom{f:=f_Z \circ g_Z = f_S \circ g_S}{X}{\PP^1}
$$ 
the abelian fibration, so we have a commutative diagram
$$
\xymatrix{
& X \ar[ld]_{g_Z} \ar[rd]^{g_S} \ar[dd]_f &
\\
Z \ar[rd]_{f_Z} & & S \ar[ld]^{f_S}
\\
& \PP^1 &
}
$$
The direct image $f_* \Omega_X$ has rank two, and
its saturation $(f^* f_* \Omega_X)^{sat} \subset \Omega_X$ coincides with
with the saturation $(g_Z^* \Omega_Z)^{sat} \subset \Omega_X$. 

Assume now that in this situation the cotangent bundle $\Omega_X$ is pseudeffective. By Theorem \ref{theorem-direct-image} we know that 
$$
\sV := (f^* f_* \Omega_X)^{sat}
$$
is pseudoeffective. Since
$$
\sV = (g_Z^* \Omega_Z)^{sat},
$$
we are tempted to believe that $\Omega_Z$ is pseudoeffective, but this might very well be false if $\Omega_S$ is pseudoeffective. Even if $\Omega_Z$
is pseudoeffective, we would have to show the nonvanishing conjecture
for isotrivial elliptic fibrations, a surprisingly difficult task, cp. \cite[Cor.6.8]{HP20}.
\end{example}

\subsection{Outlook - birational models} \label{subsection-outlook}

In Proposition \ref{proposition-dichotomy} we obtain a first step towards Conjecture \ref{conjecture-dichotomy}, assuming that the fibration $f$ has {\em irreducible fibres}.
Making this assumption has numerous technical advantages, moreover we can achieve this setup by using the (log-)MMP:
starting with a projective manifold $X_0$ and an Iitaka fibration $\holom{f_0}{X_0}{C}$ as in Question \ref{question-curve}, we can always find a birational model
$$
X_0 \dashrightarrow X_m
$$
such that $X_m$ is normal projective variety with canonical $\Q$-factorial singularities with an Iitaka fibration $\holom{f_m}{X_m}{C}$ such that all the fibres are irreducible. Indeed if $\sum_{i=1}^k a_i F_i$ is a reducible fibre, we can find an $1>\epsilon>0$
such that the pair $(X_0, \epsilon F_1)$ is klt. Since
$$
(K_{X_0}+\epsilon F_1)|_{F_1} \equiv - \frac{1}{a_1} \sum_{i=2}^k a_i F_i 
$$
is a non-zero antieffective divisor, the log-canonical divisor $K_{X_0}+\epsilon F_1$
is not nef over $C$. Thus by \cite[Thm.1.3]{Fuj11} we can run a terminating MMP
$$
\mu_1: X_0 \dashrightarrow X_1
$$
over $C$ such that $K_{X_1}+\epsilon (\mu_1)_*F_1$ is relatively trivial. Since all the contractions in this MMP are over $C$ and $f_0$ is the Iitaka fibration, all the contractions
in the MMP are $K$-trivial. Thus $X_1$ has canonical singularities and $f_1: X_1 \rightarrow C_1$ is the Iitaka fibration. Repeating the computation above we see that $\mu_1$
contracts the divisor $F_1$. Now we proceed inductively, since $f$ has only finitely many reducible fibres this terminates after finitely many steps.

By \cite[Cor.4.3]{HP20} the cotangent sheaf of the birational model $\Omega_{X_m}$
is pseudoeffective, this leads to a weak form of Question \ref{question-curve}:

\begin{question} \label{question-curve-bis}
Let $X$ be a projective manifold such that the cotangent bundle $\Omega_X$ is pseudoeffective. Assume that $K_X$ is semiample with numerical dimension one, so the Iitaka fibration is a morphism $\holom{f}{X}{C}$ onto a curve $C$.

Let $\holom{f_m}{X_m}{C}$ be a birational model of the Iitaka fibration such that
all the fibres are irreducible. Do 
we have 
$$
H^0(X, S^{[k]} \Omega^{[1]}_X) \neq 0
$$ 
for some $k \in \N$?
\end{question}

While the spaces of holomorphic symmetric forms are birational invariants of projective manifolds, this does not hold for projective varieties with canonical singularities.
Thus Question \ref{question-curve-bis} is not equivalent to Question \ref{question-curve}. Nevertheless we expect that a good understanding of the global sections of
$S^{[k]} \Omega^{[1]}_X$ allows to decide whether they lift to a resolution of singularities (cf. \cite[\S 4,(D)]{Sak79} and \cite[Prop.3.2]{BTVA22}).

{\bf Acknowledgements.} The Institut Universitaire de France and A.N.R project Karmapolis (ANR-21-CE40-0010)  provided excellent working conditions for this project.
We thank C. Gachet and M. P\v aun for helpful communications.

\section{Notation and basic results}

We work over the complex numbers, for general definitions in complex and algebraic geometry we refer to \cite{Dem12, Har77}. 
We use the terminology of \cite{Deb01} and \cite{KM98}  for birational geometry and notions from the minimal model program,
and \cite{Laz04b} for algebraic notions of positivity of vector bundles (cf. in particular \cite[Sect.6.2]{Laz04b} for the notion of $\Q$-twist). We will frequently use the divisorial Zariski decomposition of a divisor as defined in \cite{Bou04, Nak04}.

Manifolds and varieties will always be supposed to be irreducible and reduced.

\begin{definition} \label{definitiontildeq} 
Let $X$ be a projective manifold.  Then, as usual,
$$
q(X) = h^1(X,\sO_X) = h^0(X,\Omega_X) 
$$
is the irregularity of $X$. 
Further, we denote  by $\tilde q(X) $ the maximal irregularity $q(\tilde X)$, where
$\tilde X \to X$ is any finite \'etale cover. 
\end{definition}

\begin{definition} \label{def:reflexive}  \cite{Dru18}, \cite[Defn.2.1]{HP19}.
Let $X$ be a normal projective variety, 
and let $\sE$ be a reflexive sheaf on $X$.  
Then $\sE$ is pseudoeffective
if for some ample Cartier divisor $H$ on $X$ and for all $c>0$ there exist numbers $j \in \N$ and $i \in \N$ such that $i>cj$ and
$$
H^0(X, S^{[i]} \sE \otimes \sO_X(jH)) \neq 0.
$$ 
\end{definition} 

It will be convenient to characterise pseudoeffective in terms of tautological class:

\begin{definition} \label{definitiontautological}
Let $X$ be a normal variety, and let $\sE$ be a reflexive sheaf on $X$. 
\begin{itemize} 
\item 
Denote by $$ \nu: \PP'(\sE) \to \PP(\sE)$$ 
the normalization of the unique component of $\PP(\sE)$ 
that dominates $X$. 
\item Set  $\sO_{\PP'(\sE)}(1) := \nu^*(\sO_{\PP(\sE)}(1))$. 
\item 
Let $X_0 \subset X$ be the locus where $X$ is smooth and  $\sE$ is locally free, and let 
$$\holom{r}{P}{\PP'(\sE)}$$ 
be a birational morphism from a manifold $P$ such
that the complement of $\fibre{(p \circ \nu \circ r)}{X_0} \subset P $ is a divisor $D$.
\item Set $\pi := p \circ \nu \circ r$ and $\sO_{P}(1) := r^*(\sO_{\PP'(\sE)}(1))$, where $p: \PP(\sE) \to X$ is the projection.  
\item  By \cite[III.5.10.3]{Nak04} there exists an effective
divisor $\Lambda$ supported on $D$ such that
\begin{equation} \label{push-tautological}
\pi_* (\sO_{P}(m) \otimes \sO_P(m \Lambda)) \simeq S^{[m]} \sE \qquad \forall \ m \in \N.
\end{equation}
\item  We call $\zeta := c_1(\sO_{P}(1) \otimes \sO_P(\Lambda))  \in N^1(P)$
 a tautological class of $\sE$. 
\end{itemize}
\end{definition}

\begin{remark*} \label{remark-restrict-mr-general}
We will frequently restrict the reflexive sheaf to a subvariety $W \subset X$ such that
$\sE$ is locally free in an analytic neighbourhood of $W$.
In this neighbourhood
$P$ coincides with $\PP(\sE)$, so the restriction of the
tautological class $\zeta$ coincides with $\zeta_{\PP(\sE \otimes \sO_W)}$.
\end{remark*}

While we make several choices in the construction of a tautological class, its pseudoeffectivity is independent of these choices:

\begin{lemma} \label{lemmapseff} \cite[Lemma 2.7]{Dru18} \cite[Lemma 2.3]{HP19}.
Let $X$ be a normal projective variety.
Let $\sE$ be a reflexive sheaf on $X$, and let $\zeta$ be a tautological class on $\holom{\pi}{P}{X}$
(cf. Definition \ref{definitiontautological}).
Then $\zeta $ is pseudoeffective if and only if $\sE$ is pseudoeffective. 
\end{lemma}

\begin{remark}\label{anav}
In the situation of Definition \ref{definitiontautological}, 
assume that we can find a possibly singular metric $h$ on $\sO_{P}(1)$ such that $i\Theta_{h} (\sO_{P}(1)) + C_1 [D] \geq 0$ for some constant $C_1$.
Then for some $m \in \N$ the divisor $\Lambda':= C_1 [D] + m \Lambda$ is effective.
Moreover, up to replacing $m$ by some $m' \gg m$, we can assume 
that $(C_1 [D] + m \Lambda)|_\Gamma$ is not $\pi|_\Gamma$-pseudoeffective for every prime component $\Gamma \subset D$ (indeed
$\Lambda$ has this property, cf. the proof of \cite[III,Lemma 5.10.(3)]{Nak04}.
Thus $\sO_{P}(1) + \Lambda'$ is also a tautological class for $\sE$. Since it is pseudoeffective, the reflexive sheaf $\sE$ is pseudoeffective.
\end{remark}	

\begin{definition}  \label{def:Kodaira-1} Let $X$ be a normal projective variety, and let $\sE$ be a reflexive sheaf on $X$. 
Let $\holom{\pi}{P}{X}$ be as in Definition \ref{definitiontautological},
and let $\zeta$ be a tautological class on $P$.
Then we set
$$
\kappa (X, \sE) = \kappa (P, \zeta).
$$
\end{definition} 

By construction $\kappa (P, \zeta) \geq 0$ if and only if $H^0(X,S^{[m]}(\sE)) \neq 0$ for some $m \in \N$.

\begin{remark*}
If $\pi: P \to X$ denotes the projection, then 
$$ H^0(P, \sO_P(m\zeta)) = H^0(X,S^{[m]}(\sE)) $$
for all positive numbers $m$, and therefore the definition is independent on the choices made. 
\end{remark*}

\begin{lemma} \label{lemmanumericalclassfibres}
Let $X$ be a projective manifold of dimension $n$ that admits a fibration $\holom{f}{X}{C}$ onto a curve $C$.
Let $D$ be an $\R$-divisor class on $X$ such that 
$$
D^2 \cdot A^{n-2} \geq 0, \qquad D|_F \equiv 0
$$
where $A$ is an ample class on $X$ and $F$ is a general fibre of $f$. Then we have
$D \equiv \lambda F$ for some $\lambda \in \R$.
\end{lemma}

\begin{proof}
Let $i: S \subset X$ be a smooth surface cut out by general elements of some sufficiently high multiple of $A$. Then we know by the Lefschetz hyperplane theorem that the pull-back map
$$
i^* : N^1(X) \rightarrow N^1(S)
$$
is injective, so it is sufficient that the equality holds for the restrictions to $S$. 
Since the statement is trivial for $D \equiv 0$, so we can assume that $D|_S \not\equiv 0$.
The class $F|_S$
is nef, not zero and of numerical dimension one. Since $F|_S \cdot D|_S = 0$ we obtain that
by the Hodge index theorem that $D|_S^2 \leq 0$. Since by assumption 
$D^2 \cdot A^{n-2} \geq 0$ we have $D|_S^2 =0$. Thus we conclude with \cite[Lemma 2.2]{LP18} that $D|_S \equiv \lambda F|_S$ for some $\lambda \in \R$.
\end{proof}

\begin{lemma} \label{lemmapushpseff}
Let $X'$ be a normal projective variety, and let $\sE'$ be a pseudoeffective reflexive sheaf on $X'$. Let  $\holom{\mu}{X'}{X}$ be a birational morphism onto a normal projective variety $X$, and set $\sE:= (\mu_* \sE')^{[**]}$. 
Then $\sE$ is pseudoeffective.
\end{lemma}

\begin{proof}
Let $A$ be an ample Cartier divisor on $X$, so its pull-back $\mu^* A$ is nef and big.
Since $\sE'$ is pseudoeffective we know by \cite[Lemma 2.2]{HLS22}\footnote{The statement is for vector bundles, but the proof works for reflexive sheaves.} that for every $c>0$ there
exist positive integers $i,j$ such that $i>cj$ and 
$$
H^0(X', S^{[i]} \sE' \otimes \sO_{X'}(j \mu^* A)) \neq 0.
$$ 
Since the image of the exceptional locus has codimension two in $X$ and $S^{[i]} \sE$ is reflexive the projection formula yields an inclusion 
$$
H^0(X', S^{[i]} \sE' \otimes \sO_{X'}(j \mu^* A)) \hookrightarrow
H^0(X, S^{[i]} \sE \otimes \sO_{X}(j A)). 
$$
Thus $\sE$ is pseudoeffective.
\end{proof}

\begin{definition}
Let $X$ be a projective manifold, and let $\merom{f}{X}{Z}$ be a meromorphic map
onto a projective variety $Z$. Let $i: X_0 \hookrightarrow X$ be a non-empty Zariski open set such that $f|_{X_0}$ is a morphism. Then we denote by $(f^* \Omega_Z)^{sat}$ the saturation
of the coherent sheaf $i_* (f|_{X_0})^* \Omega_Z$ in $\Omega_X$.
\end{definition}

\begin{lemma} \label{lemmachangemodel}
Let $X$ be a projective manifold, and let $\merom{f}{X}{Z}$ be a meromorphic map
onto a projective variety $Z$. Assume that there exists a birational morphism
$\holom{\mu}{X'}{X}$ and a bimeromorphic model $\merom{f'}{X'}{Z'}$ of $f$
such that $((f')^* \Omega_{Z'})^{sat}$ is pseudoeffective. Then
$(f^* \Omega_{Z})^{sat}$ is pseudoeffective.
\end{lemma}

\begin{proof}
Let $X_0 \subset X$ be the maximal Zariski open set where $\mu^{-1}$ is defined, identify this set to its image in $X'$.
There exists a Zariski open subset $X_1 \subset X_0$ such that $f_{X_1}$ and $f'_{X_1}$
are morphisms. Thus $((f')^* \Omega_{Z'})^{sat}$ and $(f^* \Omega_{Z})^{sat}$ coincide
on $X_1$. Since $\Omega_X|_{X_0} \simeq \Omega_{X'}|_{X_0}$ and the saturation of a sheaf is unique, we see that  $((f')^* \Omega_{Z'})^{sat}$ and $(f^* \Omega_{Z})^{sat}$ coincide
on $X_0$. Since $X \setminus X_0$ has codimension at least two this shows 
that $(\mu_* ((f')^* \Omega_{Z'})^{sat})^{**} \simeq (f^* \Omega_{Z})^{sat}$.
Now conclude with Lemma \ref{lemmapushpseff}.
\end{proof}

\begin{lemma} \label{lemmafactorisation}
Let $X$ be a projective manifold with pseudoeffective cotangent bundle $\Omega_X$.
Assume that $X$ admits a fibration $\holom{f}{X}{C}$ such that the general fibre $F$
has numerically trivial canonical bundle.
Then there exists an almost holomorphic map $\merom{g}{X}{Z}$ onto a projective manifold $\holom{h}{Z}{C}$ such that the general fibre of $h$ is birational to a quotient of an abelian variety and $(g^* \Omega_Z)^{sat}$ is pseudoeffective.
\end{lemma}

\begin{proof}
Denote by $\holom{\pi}{\PP(\Omega_X)}{X}$ the projectivisation, and by $\zeta_X$ the tautological class.
Let $F' \rightarrow F$ be an \'etale Galois cover with group $G$ giving the Beauville-Bogomolov decomposition of $F$, i.e.  $F'$ is a product $A \times Y$ where $A$ is an abelian variety and $\pi_1(Y)=1$. 

{\em 1st case. $\dim A=0$.} In this case we set $Z=C$, so our goal is to prove that
$(f^* \Omega_C)^{sat}$ is pseudoeffective. 
By \cite[Thm.1.6]{HP19} we know that $\Omega_F$ is not pseudoeffective. Consider now the exact sequence
\begin{equation} \label{help1}
0 \rightarrow (f^* \Omega_C)^{sat} \rightarrow \Omega_X \rightarrow \sG \rightarrow 0.
\end{equation}
Since $\sG$ is torsion-free, the locus where the fibration $\PP(\sG) \rightarrow X$
is not equidimensional has codimension at least two. Since $\PP(\sG) \subset \PP(T_X)$,
every fibre has dimension at most $\dim X-1$. Thus $\PP(\sG) \subset \PP(T_X)$
has a unique divisorial component $E$ and
$$
[E] = \zeta - \pi^* (f^* \Omega_C)^{sat}.
$$
Since for a general fibre
$$
\sG \otimes \sO_F \simeq \Omega_F
$$
is not pseudoeffective, the restriction of the tautological class $\zeta_X$ to $E$ is not pseudoeffective. Since $\zeta_X$ is pseudoeffective, there exists thus a $\lambda>0$
such that the restriction of $\zeta_X-\lambda E$ to $E$ is pseudoeffective. We claim that $\lambda \geq 1$: let $\holom{\pi_E}{E}{X}$ be the induced fibration, and set $E_F:=\fibre{\pi_E}{F}$ for $F$ a very general fibre. Then
$(\zeta_X-\lambda E)_{E_F}$ is pseudoeffective. Yet the restriction of \eqref{help1} to
$F$ is simply
$$
0 \rightarrow \sO_F \rightarrow \Omega_X \otimes \sO_F \rightarrow \Omega_F \rightarrow 0,
$$
so $-E|_{E_F} = -E_F|_{E_F} = \zeta_{\PP(\Omega_F)}$ yields 
$$
(\zeta_X-\lambda E)_{E_F} = (1-\lambda) \zeta_{\PP(\Omega_F)}.
$$
Since $\zeta_{\PP(\Omega_F)}$ is not pseudoeffective, we have $\lambda \geq 1$.
Recalling that  $\zeta -[E] = \pi^* (f^* \Omega_C)^{sat}$ we obtain the statement.

{\em 2nd case. $\dim A>0$.}
By \cite[Lemma, p.8]{Katata} the group $G$ acts diagonally on the product
$A \times Y$, so $F$
admits a fibration $F \rightarrow A/G$ onto a normal projective variety $A/G$. 
Since $F$ is a general fibre, the fibres of this fibration determine an irreducible component $Z \subset \Chow{X/C}$ such that the natural morphism $X \dashrightarrow Z$ is almost holomorphic and holomorphic when restricted to $F$.
Up to replacing $Z$ by a desingularisation we can assume that $Z$ is smooth,
moreover by Lemma \ref{lemmachangemodel} it is sufficient to prove 
that $((g')^* \Omega_Z)^{sat}$ is pseudoeffective for some birational model $g'$
of the fibration $g$. Since the property of having pseudoeffective cotangent bundle
is a birational invariant for projective manifolds \cite[Prop.4.1]{HP20} we can resolve the
indeterminacies of $g$. 
In order to simplify the notation we denote the birational models by the same letters
and denote the natural maps by $\holom{g}{X}{Z}$ and $\holom{h}{Z}{C}$.
By construction the general $g$-fibre $G$ is an \'etale quotient of $Y$, 
so by \cite[Thm.1.6]{HP19} its cotangent bundle is not pseudoeffective. 

Let
\begin{equation} \label{exactfibration}
0 \rightarrow (g^* \Omega_Z)^{sat} \rightarrow \Omega_X \rightarrow \sG \rightarrow 0
\end{equation}
be the exact sequence defined by the fibration $g$. 
Let $\holom{\mu}{\Gamma}{\PP(\Omega_X)}$ be the 
composition of an embedded resolution of $\PP(\sG) \subset  \PP(\Omega_X)$
and the blowup of the unique irreducible component of $\PP(\sG)$ that surjects onto $X$,
and denote by $E$ the exceptional locus. 
Then we have a natural fibration $\holom{\psi}{\Gamma}{\PP((g^* \Omega_Z)^{sat})}$
such that $\mu^* \zeta_X - E \simeq \psi^* \zeta_Z$ where
$\zeta_Z$ is the tautological class on $\PP((g^* \Omega_Z)^{sat})$.
We summarise the construction in a commutative diagram:
$$
\xymatrix{
\Gamma \ar[rd]^{\psi} \ar[d]_\mu & 
\\
\PP(T_X) \ar[d]_\pi & \PP((g^* \Omega_Z)^{sat}) \ar[ld]
\\
X \ar[r]^g \ar[d]_f & Z \ar[ld]_h
\\
C &  
}
$$
We are done if we show that $\mu^* \zeta_X - E$ is pseudoeffective: 
since the restriction of $\zeta_X$ to $\PP(\sG)$ is not pseudoeffective, we know that 
the restriction of $\mu^* \zeta_X$ to $E$ is not pseudoeffective. Thus there exists a positive number $\lambda>0$ such that the restriction of $\mu^* \zeta_X - \lambda E$ 
to $E$ is pseudoeffective. As in the first case we prove the statement by showing that $\lambda \geq 1$: let $z \in Z$ be a very general point and $G$ the $g$-fibre over it, then the restriction of \eqref{exactfibration} to $G$ is
$$
0 \rightarrow \sO_{G}^{\oplus \dim Z} \rightarrow \Omega_X \otimes \sO_G \rightarrow \Omega_G \rightarrow 0.
$$
Thus the fibre of $E \rightarrow \PP(\sG) \rightarrow X \rightarrow Z$ over $z$ is simply
$$
E_z \simeq \PP(N^*_{\PP(\Omega_G)/\PP(\Omega_X \otimes \sO_G)}
\simeq \PP(\sO_{\PP(\Omega_G)}(-1)^{\oplus \dim Z}) \simeq \PP(\Omega_G) \times \PP^{\dim Z-1}.
$$
and 
$$
-E|_{E_z} = \mu^* \zeta^*_{\PP(\Omega_G)} + p_{\PP^{\dim Z-1}}^* H
$$
with $H$ the hyperplane class. Thus we have
$$
(\mu^* \zeta_X - \lambda E)|_{E_z} = (1-\lambda)  \mu^* \zeta^*_{\PP(\Omega_G)}
+ \lambda p_{\PP^{\dim Z-1}}^* H.
$$
Since $E_z$ is a product, we finally obtain that $\lambda \geq 1$.
\end{proof}

\begin{remark}
As one can see from the construction of $Z$ in the proof of Lemma \ref{lemmafactorisation} we can assume that the general fibre of $Z \rightarrow C$
is {\em isomorphic} to the quotient of an abelian variety, if we accept that $Z$ is only normal. By \cite[Thm.1.2]{Shi21} a quotient of an abelian variety is a quasi-\'etale cover
of an abelian variety and a $\Q$-Fano variety. For the latter the tangent sheaf 
is generically ample \cite[Thm.1.3]{Pet12}, so its cotangent sheaf is not pseudoeffective.
Thus repeating the same argument, we obtain a factorisation
$Z \rightarrow C$ such that the general fibre is birational to a {\em quasi-\'etale}
quotient of an abelian variety.
\end{remark}

\begin{lemma} \label{lemma-splitting-image}
Let $X$ be a projective manifold of dimension $n$ such that $q(X)=0$, and 
let $\holom{f}{X}{\PP^1}$ be a fibration.
Let $T_1 \subset f_* \Omega_X$ be the saturation of  $\omega_{\PP^1} \rightarrow f_* \Omega_X$. Then we have
$$
f_* \Omega_X \simeq T_1 \oplus T_2.
$$
\end{lemma}

Note that while the lemma works without any assumption on the general fibre, the statement is trivial unless a general fibre $F$ satisfies $q(F) \neq 0$. Indeed
$\rk f_* \Omega_X \leq 1 + q(F)$, so $f_* \Omega_X \simeq T_1$ if $q(F)=0$.

\begin{proof}
Recall first that $f_* \Omega_X$ is a vector bundle, since a torsion-free sheaf on a curve is locally free. For the same reason the quotient sheaf $T_2 := \Omega_X/T_1$ is locally free, so we have an exact sequence of vector bundles
$$
0 \rightarrow T_1 \rightarrow f_* \Omega_X \rightarrow T_2 \rightarrow 0,
$$
our goal is to show that the sequence splits. Since $h^0(\PP^1, f_* \Omega_X)=q(X)=0$
and $\omega_{\PP^1} \rightarrow T_1$ is non-zero, we have $T_1 \simeq \sO_{\PP^1}(a)$ with
$a=-1$ or $a=-2$. We claim that the map 
$$
\iota: H^1(\PP^1, T_1) \rightarrow H^1(\PP^1, f_* \Omega_X)
$$
is injective. Assuming this for the time being, let us conclude: since
$h^0(\PP^1, f_* \Omega_X)=0$, the injectivity implies that $h^0(\PP^1, T_2)=0$.
Since $T_2$ is a vector bundle on $\PP^1$, it is thus antiample. 
Since $a=-1$ or $a=-2$, we easily deduce that 
$$
H^1(\PP^1, T_1 \otimes T_2^*) = 0.
$$
Thus the exact sequence splits.

{\em Proof of the claim.} If $a=-1$ we have $H^1(\PP^1, T_1)=0$ and the statement is trivial. Thus we can assume $a=-2$ and hence $T_1 \simeq \omega_{\PP^1}$. In this case  
$$
\iota: H^1(\PP^1, T_1) = H^1(\PP^1, \omega_{\PP^1}) \rightarrow H^1(\PP^1, f_* \Omega_X)
$$
composed with the natural map $H^1(\PP^1, f_* \Omega_X) \rightarrow H^1(X, \Omega_X)$
identifies to the pull-back
$$
f^*: H^1(\PP^1, \omega_{\PP^1}) \rightarrow H^1(X, \Omega_X).
$$
Since $X$ is K\"ahler, the pull-back $f^*$ is injective. This implies  the injectivity of $\iota$. 
\end{proof}

\begin{example} \label{examplepushisotrivial}
Let $C$ be a hyperelliptic curve of genus $g \geq 2$ and denote by $i_C$ the hyperelliptic involution.
Let $E$ be an elliptic curve and denote by $i_E$ the involution determined by the degree two cover
$E \rightarrow \PP^1$. We denote by $X$ the minimal resolution of $(C \times E)/\langle i_C \times i_E\rangle$, then $X$ admits a relatively minimal isotrivial elliptic fibration
$$
f : X \rightarrow \PP^1 \simeq C/\langle i_C \rangle.
$$
The elliptic fibration $f$ has exactly $2g+2$ singular fibres, all of Kodaira type $I_0^*$. In particular we have $K_X \simeq f^* M$ with $M$ a line bundle on $\PP^1$.

Looking at the action of the involution near the fixed points we see that 
$$
H^0(X, \Omega_X)=0, \qquad H^0(X, K_X) \simeq \C^g,
$$
so we have $\chi(\sO_X) = 1 + g$. By \cite[V, Prop.12.2]{BHPV04} we have
$$
\deg f_* \omega_{X/\PP^1} = g+1,
$$
in particular $f_* \omega_X$ is an ample line bundle.

The direct image $f_* \Omega_X$ behaves quite differently: since
$h^0(X, f_* \Omega_X) = q(X)=0$, the rank two vector bundle
$f_* \Omega_X$ is anti-ample.
In fact we claim that
$$
f_* \Omega_X \simeq \sO_{\PP^1}(-2) \oplus \sO_{\PP^1}(-m)
$$
with $m \geq g+1$. Thus the direct factor $T_2$ in Lemma \ref{lemma-splitting-image} can be arbitrarily negative.
\end{example}

\begin{proof}
Recall that all the singular fibres are of type $I_0^*$. Thus we have an exact sequence
$$
0 \rightarrow f^* \omega_{\PP^1} (D) \rightarrow \Omega_X \rightarrow \sI_Z \otimes \omega_{X/\PP^1}(-D) \rightarrow 0
$$
where for each singular fibre $D$ is the unique multiple component (the handle of the comb) and $Z$ is the intersection of the handle with the four teeth of the comb.
Pushing forward to $\PP^1$ we get a sequence
$$
0 \rightarrow \omega_{\PP^1} \rightarrow f_* \Omega_X \rightarrow f_* (\sI_Z \otimes \omega_{X/\PP^1}(-D))
$$
and the morphism 
$$
f_* \Omega_X \rightarrow f_* (\sI_Z \otimes \omega_{X/\PP^1}(-D))
$$
is not zero since $f$ is isotrivial. Note that $f_* (\sI_Z \otimes \omega_{X/\PP^1}(-D))$ is 
a torsion-free sheaf of rank one on $\PP^1$, so it is isomorphic to $\sO_{\PP^1}(d)$ for some $d \in \Z$. We are done if we show that $d \leq -(g+1)$:
since  $f_* (\sI_Z \otimes \omega_{X/\PP^1}(-D)) \subset f_* (\sI_Z \otimes \omega_{X/\PP^1})$ it is sufficient to show the statement for the latter.
Since $\omega_X \simeq f^* M$, we can apply the projection formula to obtain
$$
f_* (\sI_Z \otimes \omega_{X/\PP^1}) \simeq f_* (\sI_Z) \otimes f_* \omega_{X/\PP^1}.
$$ 
Since $Z$ is a finite set, we have
$$
 f_* (\sI_Z) \subset \sI_{f(Z)} \subset \sO_{\PP^1}.
$$ 
The set $f(Z)$ consists of $2g+2$ points and
$f_* \omega_{X/\PP^1}$ has degree $g+1$, this shows the statement.
\end{proof}

\begin{example}
The cotangent bundle of the manifold $X$ appearing
Example \ref{examplepushisotrivial} is not pseudoeffective
\cite[Thm.6.7]{HP20}, but we can modify the construction by replacing the hyperelliptic curve by a curve that 
is ramified double cover $C \rightarrow Y$ of an elliptic curve $Y$. The construction then leads to an isotrivial elliptic fibration $X \rightarrow Y$ having $2g-2$ singular fibres of type $I_0^*$. Applying the argument above we obtain that 
$$
f_* \Omega_X \simeq \Omega_Y \oplus L
$$ 
with $L$ a line bundle of degree at most $-g$.
\end{example}

\section{Extension of $\Omega_F$ and pseudoeffective subsheaves} 

In this section we consider vector bundles on a compact K\"ahler manifold  $X$ that are given by an extension of trivial line bundles
$$
	0\rightarrow \mathcal{O}_X \rightarrow V \rightarrow \mathcal{O}_X  ^{\oplus r} \rightarrow 0 .
$$
The extension class corresponding to the exact sequence is
an element
$$
\eta=(\eta_1, \ldots, \eta_{r}) \in H^1(X, (\mathcal{O}_X  ^{\oplus r})^*) 
\simeq H^1(X, \mathcal{O}_X)^{\oplus r}. 
$$
Looking at the coboundary map $H^0(X, \mathcal{O}_X  ^{\oplus r}) \rightarrow
H^1(X, \sO_X)$ we see that $\eta_1, \ldots, \eta_{r}$ are linear independent
 in $H^1(X, \mathcal{O}_X)$ if and only if $h^0(X, V)=1$.

In general we have $h^0(X, V)=1+ a$ for some $0 \leq a \leq r$, so the
sections of $V$ induce another extension 
$$
0\rightarrow \mathcal{O}_X ^{\oplus 1+a}\rightarrow V \rightarrow \mathcal{O}_X  ^{\oplus r-a} \rightarrow 0 .
$$
 Since $V$ is an extension of trivial vector bundles, we know that $\mathcal{O}_{\mathbb P (V)} (1)$ is a nef line bundle. There exists thus a possibly singular metric $h$ on $\mathcal{O}_{\mathbb P (V)} (1)$ such that $\frac{i}{2\pi} \Theta_h (\mathcal{O}_{\mathbb P (V)} (1)) \geq 0$. The aim of this section is to give a complete description of the metric $h$.

\begin{theorem}\label{genDPS}
	Let $X$ be a compact K\"ahler manifold and let $V$ be an extension
	$$
	0\rightarrow \mathcal{O}_X \rightarrow V \rightarrow \mathcal{O}_X  ^{\oplus r} \rightarrow 0 ,
	$$
such that $\dim H^0(X, V)=1+a$ for some $a\in \N$. The inclusion $	\mathcal{O}_X ^{\oplus 1+a} \rightarrow V$ induces a quotient morphism
\begin{equation}\label{11}
V^* \rightarrow 	(\mathcal{O}_X ^{\oplus 1+a})^\star   .
	\end{equation}
	Let $h$ be a possibly singular metric  on $\mathcal{O}_{\mathbb P (V)}(1)$ such that $i\Theta_h (\mathcal{O}_{\mathbb P (V)}(1)) \geq 0$. 
	Set $\varphi:= \log h^*$, which is a function 
	$$\varphi : V^* \rightarrow \mathbb R$$
	defined on $V^*$.
		\begin{itemize}
			\item Then there exists psh function $\psi$ on $\mathbb C^{1+a}$ such that $\varphi =  \psi \circ pr$, where $pr$ is the natural map $V^* \rightarrow (\mathcal{O}_X ^{\oplus 1+a})^* = X\times \mathbb C^{1+a} \rightarrow \mathbb C^{1+a}$.
			
		\item If $a =0$, then $\frac{i}{2\pi}  \Theta_h (\mathcal{O}_{\mathbb P (V)}(1))$ is the current of 
		integration along the divisor $\PP(\mathcal{O}_X  ^{\oplus r}) \subset \PP(V)$.
		\item If $a \geq 1$, let $\tau: \mathbb P (V) \dashrightarrow \mathbb P  (\mathcal{O}_X ^{\oplus 1+a}) $ be the rational map induced by \eqref{11}.
		Then there exists a semipositive current $T$ on $\mathbb P^{a}$ such that  $\frac{i}{2\pi} \Theta_h (\mathcal{O}_{\mathbb P (V)}(1)) = \tau^* pr^* T$\footnote{By Remark \ref{remark-projection} the projection $\tau$ is obtained as the composition of 
a blowup $\holom{\mu}{\widehat \PP(V)}{\PP(V)}$ and a smooth fibration $\holom{q}{\widehat \PP(V)}{\mathbb P  (\mathcal{O}_X ^{\oplus 1+a})}$. In particular the current
$$
\tau^* pr^* T := \mu_* q^* pr^* T
$$
is well-defined.}, 
where $pr$ is the projection $\mathbb P  (\mathcal{O}_X ^{\oplus 1+a}) = X\times \mathbb P^{a} \rightarrow  \mathbb P^{a}$.
	\end{itemize}
\end{theorem}

Note that the case $a=0$ corresponds to the statement of Theorem \ref{theorem-dpsgen}. Before proving Theorem \ref{genDPS}, we need some preparations. 

\begin{remark} \label{remark-projection}
For the convenience of the reader let us recall some basic facts about projections from a quotient bundle. Let
\begin{equation} \label{extension-one}
0 \rightarrow K \rightarrow V \rightarrow Q \rightarrow 0
\end{equation}
be an extension of vector bundles over a manifold $X$, so that we have an inclusion 
of projectivised bundles $\PP(Q) \subset \PP(V)$. Denote by
$\holom{\pi_K}{\PP(K)}{X}$ the projectivisation of $K$.

Let $\holom{\mu}{\widehat \PP(V)}{\PP(V)}$
be the blow-up of $\PP(V)$ along $\PP(Q)$, and denote by $E$ the exceptional divisor. We have a fibration
$$
\holom{q}{\widehat \PP(V)}{\PP(K)}
$$
such that $\mu^* \sO_{\PP(V)}(1) \simeq q^* \sO_{\PP(K)}(1) \otimes \sO_{\widehat \PP(V)}(E)$.
Moreover we have  $\widehat \PP(V) \simeq \PP(W)$ where $W \rightarrow \PP(K)$
is a vector bundle given by an extension
\begin{equation} \label{extension-two}
0 \rightarrow  \sO_{\PP(K)}(1)  \rightarrow W := q_*  \mu^* \sO_{\PP(V)}(1) \rightarrow \pi_K^* Q \rightarrow 0.
\end{equation}
The extension splits on every fibre of $\pi_K$, but its global extension class
is in 
$$
H^1(\PP(K), \sO_{\PP(K)}(1) \otimes \pi_K^* Q^*) \simeq H^1(X, K \otimes Q^*).
$$
In fact the extension classes of \eqref{extension-one} and \eqref{extension-two}
coincide. 
\end{remark}

\begin{remark} \label{remark-numerically-flat}
In the situation of Theorem \ref{genDPS} the vector bundle $V$ is numerically flat, thanks to \cite{Sim92, Den21}, we know that $V$ is a local system. We consider the universal cover $\sigma: \widetilde{X} \rightarrow X$ and  obtain the diagram	
\begin{equation}\label{comdia}
		\begin{CD}
\mathbb P (\sigma^* V) \simeq \widetilde{X}  \times \mathbb P^{r}@>{\sigma_V}>> \mathbb P (V) \\
	@V \tilde \pi VV      @VV\pi V  \\ 
	\widetilde{X} @>>{\sigma}>   X
\end{CD}	
\end{equation}
where the isomorphism $\mathbb P (\sigma^* V) \simeq \widetilde{X}  \times \mathbb P^{r}$ is induced by the flat connection of the local system $V$. Fix a point $x\in X$ and a base of $V_x$. The flat connection induces a representation
\begin{equation}\label{grourep}
	\rho: \pi_1 (X) \to GL (r+1,\mathbb C) .
	\end{equation}
\end{remark}

With respect to this isomorphism, we have the following lemma for the divisorial part of $i\Theta_h (\mathcal{O}_{\mathbb P (V)}(1)) $.

\begin{lemma}\label{singlocus}
	Let $X$ be a compact K\"ahler manifold and let $V$ be a rank $r$ numerically flat vector bundle on $X$. Let $h$ be a possibly singular metric on $\mathcal{O}_{\mathbb P (V)} (1)$ such $i\Theta_h (\mathcal{O}_{\mathbb P (V)} (1)) \geq 0$. We consider  Siu's decomposition 
	$$
	i\Theta_h (\mathcal{O}_{\mathbb P (V)} (1)) = \sum a_i E_i + \Theta_0 ,
	$$
	where $E_i$ are divisors, and $\Theta_0\geq 0$ is a current where the locus of  positive Lelong number is of codimension at least $2$.
	
Following the notations in the Diagram \eqref{comdia},	if $\pi (E_i) =X$, then there exists a hypersurface $Z\subset \mathbb P^{r-1} $ such that 
	$$\sigma_V ^{-1} (E_i) = \pi_2 ^{-1} (Z) ,$$ 
	where $\pi_2 :  \widetilde{X}  \times \mathbb P^{r } \rightarrow \mathbb P^{r} $ is the natural projection. In particular, $c_1 (E_i) = \lambda_i c_1 (\mathcal{O}_{\mathbb P (V)} (1))$ for certain constant $\lambda_i \geq 0$.
\end{lemma}	

\begin{proof}
	Let $m, m_1$ be two arbitrary positive integers such that $m \geq m_1$. 
	We have the natural inclusion
	\begin{equation}\label{inclu}
		\pi_* (K_{\mathbb P (V) /X} \otimes \mathcal{O}_{\mathbb P (V)} (m) \otimes \mathcal{I} (m_1 h)) \subset \pi_* (K_{\mathbb P (V) /X} \otimes \mathcal{O}_{\mathbb P (V)} (m)) .
		\end{equation}
	By assumption, $\pi_* (K_{\mathbb P (V) /X} \otimes \mathcal{O}_{\mathbb P (V)} (m))$ is numerically flat. 
	
	Since $\mathcal{O}_{\mathbb P (V)} (1)$ is nef and relatively ample, for every $\ep >0$, we can find a smooth metric $h_\ep$ on $\mathcal{O}_{\mathbb P (V)} (1)$ such that $i\Theta_{h_\ep} (\mathcal{O}_{\mathbb P (V)} (1)) \geq -\ep \pi^* \omega_X $. 
	Now we equip $\mathcal{O}_{\mathbb P (V)} (m)$ with the metric $h_\ep ^{\otimes (m-m_1)} \cdot h^{\otimes m_1}$. It induces a $L^2$ metric $H_\ep$ on $\pi_* (K_{\mathbb P (V) /X} \otimes \mathcal{O}_{\mathbb P (V)} (m) \otimes \mathcal{I} (m_1 h))$ and by \cite{PT18} we know that 
	$$i\Theta_{H_\ep} (\pi_* (K_{\mathbb P (V) /X} \otimes \mathcal{O}_{\mathbb P (V)} (m) \otimes \mathcal{I} (m_1 h))) \succeq -\ep \omega_X .$$
	In particular, $\pi_* (K_{\mathbb P (V) /X} \otimes \mathcal{O}_{\mathbb P (V)} (m) \otimes \mathcal{I} (m_1 h))$  is weakly positively curved cf. \cite[Definition 2.1]{CCM21}.
	Together with the inclusion \eqref{inclu} and the fact that $c_1 (\pi_* (K_{\mathbb P (V) /X} \otimes \mathcal{O}_{\mathbb P (V)} (m)))=0$, we know that 
	$$c_1 (\pi_* (K_{\mathbb P (V) /X} \otimes \mathcal{O}_{\mathbb P (V)} (m) \otimes \mathcal{I} (m_1 h)))=0.$$ 
By using \cite[Corollary, Page 3]{Wu22}, we know that $\pi_* (K_{\mathbb P (V) /X} \otimes \mathcal{O}_{\mathbb P (V)} (m) \otimes \mathcal{I} (m_1 h))$ is numerically flat and its flat connection is compatible with the flat connection of $V$.
	
	Now for $m\gg m_1 \gg 1$, we know that $\pi_* (K_{\mathbb P (V) /X} \otimes \mathcal{O}_{\mathbb P (V)} (m) \otimes \mathcal{I} (m_1 h))$ are fiberwisely some $(m-r+1)$-degree polynomials vanishing along $\sum E_i$. The numerically flatness implies that these polynomials are locally of constant coefficients. Therefore $E_i$ are invariant with respect to the flat connection. 
	\end{proof}

The following lemma is elementary, we give the details for the convenience of the reader:

\begin{lemma}\label{newbase}
		Let $X$ be a compact K\"ahler manifold and let $V$ be an extension
	\begin{equation}\label{ext1}
		0\rightarrow \mathcal{O}_X \rightarrow V \rightarrow \mathcal{O}_X ^{\oplus r} \rightarrow 0 .
	\end{equation}
We suppose that $\dim H^0 (X, V) =1+a$ for some $a \geq 0$.	Then the representation $\rho: \pi_1 (X) \to GL (r+1, \mathbb C)$ corresponding to the flat connection of $V^*$ can be taken as follows 
\[
\rho (\Lambda) = \begin{bmatrix}
	1 & 0 & 0 & \cdots& 0 & 0& b_{1} (\Lambda)\\
	0  & 1 & 0 & \cdots &  0 &0& \vdots \\
	\vdots & \vdots &  1 & 0& \ddots & 0 & b_{r-a} (\Lambda) \\
	\vdots & \vdots &  0 & 0 & \ddots& 0 & 0 \\
		\vdots & \vdots &  0 & 0& \ddots& 0 & \cdots \\
		\vdots & \vdots &  0 & 0& \ddots& 1 & 0 \\
	0 & \cdots & \cdots & \cdots & 0 & 0& 1
\end{bmatrix}  \qquad\text{ for } \Lambda\in \pi_1 (X).
\]
Moreover, the $\mathbb C$-linear combination of the vectors $\{ (b_1 (\Lambda), \cdots, b_{r-a} (\Lambda))\}_{\Lambda\in \pi_1 (X)}$ generates the total space $\mathbb C^{r-a}$.
\end{lemma}	

\begin{proof}
Let $\pi: \widetilde{X} \rightarrow X$ be the universal cover. We fix a point $x\in X$.
Since $V^*$ is an extension of $\mathcal{O}_X$ by $\sO_X ^{\oplus r}$, by using \cite[Proof of Thm A]{Den21}, we can find a base $\{e^*_i\}_{i=1}^{r+1}$ of $V^* _x$ such that the representation of induced by the flat connection of $V^*$ under the base $\{e^*_i\}_{i=1}^{r+1}$  is of form 
\[
\rho (\Lambda) = \begin{bmatrix}
	1 & 0 & 0 & \cdots& b_{1} (\Lambda)\\
	0 & 1 & 0 & \cdots & b_2 (\Lambda)\\
	\vdots & 0 & 1 & \cdots & \cdots \\
	\vdots & \vdots &  0 & \ddots & b_{r-1} (\Lambda) \\
	\vdots & \vdots &  0 & \ddots & b_r (\Lambda) \\
	0 & \cdots & \cdots & \cdots & 1 
\end{bmatrix}  \qquad\text{ for } \Lambda\in \pi_1 (X).
\]
Namely, 
$$\rho (\Lambda) (e_i ^*) = e_i ^*  \qquad\text{for every }  i \leq r ,$$
and 
$$\rho (\Lambda) (e_{r+1} ^*) = e_{r+1} ^*  +\sum_{i\leq r} b_i (\Lambda) e^* _i  .$$
\medskip

Now let $\{e_i\}_{i=1}^{r+1}$ be the dual base of $\{e^*_i\}_{i=1}^{r+1}$.  
For any collection  $a_1, \ldots, a_{r+1} \in\mathbb \C$, the parallel translation of $\sum_{i=1}^{r+1} a_i e_i$ induces a flat section of $H^0 (\widetilde{X}, \pi^* V)$. By abusing a bit the notation, we still denote the induced flat section by $\sum_{i=1}^{r+1} a_i e_i$.
Now the flat section 
$$\sum_{i=1}^{r+1} a_i e_i \in H^0 (\widetilde{X}, \pi^* V)$$
comes from a section of $V$ if and only if for any $\lambda_1, \ldots, \lambda_{r+1} \in \C$ and any $\Lambda \in \pi_1 (X)$, we have
$$\langle \sum_{i=1}^{r+1} a_i e_i , \sum_{i=1}^{r+1} \lambda_i\ e^*_i  \rangle = \langle \sum_{i=1}^{r+1} a_i e_i , \rho (\Lambda)(\sum_{i=1}^{r+1} \lambda_i\ e^*_i ) \rangle $$
$$= \langle \sum_{i=1}^{r+1} a_i e_i , \sum_{i\leq r} \lambda_i e^*_i  + \lambda_{r+1}\cdot  (\sum_{i\leq r} b_i (\Lambda) e^* _i + e_{r+1} ^*) \rangle  .$$
Namely
$$\sum_{i=1}^{r+1} a_i \lambda_i = \sum_{i=1}^{r+1} a_i \lambda_i + \lambda_{r+1} \cdot (\sum_{i\leq r} b_i (\Lambda) a_i) \qquad\text{for every } \lambda_i \in \C$$
Therefore $\sum_{i=1}^{r+1} a_i e_i$ comes from a section of $H^0 (X, V)$ if and only if
\begin{equation}\label{flats}
	\sum_{i\leq r} b_i (\Lambda) a_i =0 \qquad\text{for any } \Lambda \in \pi_1 (X). 
\end{equation}

Notice that \eqref{flats} implies that  $e_{r+1}$ comes from a section of $H^0 (X, V)$. It corresponds to the trivial subline bundle in the exact sequence \eqref{ext1}. 

\medskip
By assumption, $\dim H^0 (X, V) =1+a$. Then $e_{r+1}$ is one section and 
we can find another $a$ linearly independent flat sections which we can write :  for $1\leq j\leq a$,
\begin{equation}\label{defa}
	s_j := \sum_{1\leq i\leq r} a_{j,i} e_i
\end{equation}
coming from $V$.  By using \eqref{flats},  we have
\begin{equation}\label{eqq}
	\sum_{1\leq i \leq r}  a_{j,i} b_i (\Lambda)=0 \qquad\text{for any } \Lambda \in \pi_1 (X), 1\leq j \leq a .
\end{equation}

Since $\{s_j\}$ are linear independent, we can find an invertible $r\times r$-matrix $A=(a_{j,i})$ such
that the first $a$ lines are given by the vector $(a_{j,1}, \ldots, a_{j,r})$ for $1\leq j\leq a$.
\medskip

Now we choose a new base $\widetilde{e}^*_{i}$ of  $V^* _x$:
$$e_{r+1} ^* = \widetilde{e}^*_{r+1} \qquad\text{and} \qquad e_i ^* = \sum_{1\leq j\leq r} a_{j,i } \widetilde{e}^*_{j} \text{ for } 1 \leq i\leq r .$$
For this new base, we still have
$$\rho (\Lambda) \widetilde{e}^*_{i} = \widetilde{e}^*_{i} \qquad\text {for every }  i\leq r.$$
Moreover, 
$$\rho (\Lambda) \widetilde{e}^*_{r+1} = \rho (\Lambda) (e_{r+1} ^*) =e_{r+1} ^*+ \sum_{i\leq r} b_i (\Lambda) e^* _i $$
$$= \widetilde{e}^*_{r+1} + \sum_{1\leq j\leq r}\big(\sum_{1\leq i\leq r} b_i (\Lambda) a_{j,i} \big) \cdot  \widetilde{e}^*_{j} .$$
Thanks to \eqref{eqq} we know that 
$$
\sum_{i\leq r} b_i (\Lambda) a_{j,i} =0 \qquad\text{for every } 1\leq j \leq a .$$
Therefore
$$\rho (\Lambda) \widetilde{e}^*_{r+1} = \sum_{a+1 \leq j\leq r} c_j (\Lambda )\widetilde{e}^*_{j} + \widetilde{e}^*_{r+1} .$$
After changing the orders of $\widetilde{e}^*_{j}$,  the representation $\pi_1 (X) \to GL (r+1, \mathbb C)$ under the new base is of the form 
\[
\rho (\Lambda) = \begin{bmatrix}
	1 & 0 & 0 & \cdots& c_{1} (\Lambda)\\
	0 & 1 & 0 & \cdots & c_2 (\Lambda)\\
	\vdots & 0 & 1 & \cdots & c_{r-a} (\Lambda)\\
	\vdots & \vdots &  0 & \ddots & 0 \\
	\vdots & \vdots &  0 & \ddots & \cdots  \\
	0 & \cdots & \cdots & \cdots & 1 
\end{bmatrix}  \qquad\text{ for } \Lambda\in \pi_1 (X).
\]	
\medskip

Finally, if the space generated by $\{(b_1 (\Lambda), \cdots , b_{r-a} (\Lambda)) \}_{\Lambda\in \pi_1 (X)} $ is a strict subspace in $\mathbb C^{r-a}$, then we can find some $c_i\in \mathbb C$ such that $\sum_{i\leq r-a} c_i b_i (\Lambda) =0$ for every $\Lambda \in \pi_1 (X)$.
By using \eqref{flats}, we know that $\sum_{i\leq r-a} c_i  e_i \in H^0 (X, V)$ and this section is linearly independent from the sections constructed above.  This contradicts $\dim H^0 (X, V)= a+1$.
\end{proof}	

The following lemma is a special case of Theorem \ref{genDPS} and will not be needed for its proof. However we find it quite instructive to see how the representation allows to directly prove the algebraic counterpart of the main theorem.

\begin{lemma}\label{uniquesects}
	Let $X$ be a compact K\"ahler manifold and let $V$ be an extension
	\begin{equation}\label{ext}
		0\rightarrow \mathcal{O}_X \rightarrow V \rightarrow \mathcal{O}_X ^{\oplus r} \rightarrow 0 .
	\end{equation}
	Then $H^0 (X, S^k V) = S^k H^0 (X, V)$ for every $k \in \N$.
\end{lemma}

\begin{proof}
	We have $\dim H^0 (X, V) = 1+a$ for some $0 \leq a \leq r$. Fix a point $x_0 \in X$ and the base $\{e^*_i\}_{i=1}^{r+1}$ of $V^* _{x_0}$ in Lemma \ref{newbase}. Then the representation of the fundamental group corresponding to the local system $V^*$ under this base becomes 
\[
\rho (\Lambda) = \begin{bmatrix}
	1 & 0 & 0 & \cdots& 0 & 0& b_{1} (\Lambda)\\
	0  & 1 & 0 & \cdots &  0 &0& \vdots \\
	\vdots & \vdots &  1 & 0& \ddots & 0 & b_{r-a} (\Lambda) \\
	\vdots & \vdots &  0 & 0 & \ddots& 0 & 0 \\
	\vdots & \vdots &  0 & 0& \ddots& 0 & \cdots \\
	\vdots & \vdots &  0 & 0& \ddots& 1 & 0 \\
	0 & \cdots & \cdots & \cdots & 0 & 0& 1
\end{bmatrix}  \qquad\text{ for } \Lambda\in \pi_1 (X).
\]
	
Let $s\in H^0(X, S^k V)$. Then $s(x_0)$ can be seen as a degree $k$-homogeneous  polynomial $P$ on $ V^* _{x_0} =\mathbb C^{r+1}$.  By using \cite[Lemma 4.1]{Cao15}, we know that $s$ is flat.
	In particular the polynomial $P$  is invariant with respect to the $\rho$-action, namely 
	\begin{equation}\label{relation11}
		P (w_1,  \cdots, w_{r+1}) = P (w_1 + b_1 (\Lambda) w_{r+1},  \cdots,
		w_{r-a} + b_{r-a} (\Lambda) w_{r+1},  w_{r-a+1} ,\cdots,  w_{r+1} )
	\end{equation}
	for any  $\Lambda \in \pi_1 (X)$, where $(w_1, \cdots, w_{r+1}) \in V^* _{x_0} $ 
	are coefficients with respect to the base $\{e^*_i\}_{i=1}^{r+1}$.
	
	To prove the lemma, it is equivalent to show that $P (w)$ depends only on $w_{r-a+1}, \cdots, w_{r+1}$. We prove it by induction on the degree of $P$.
	\medskip
	
	If $k=1$, since $\dim H^0 (X, V)=a+1$, we know that the linear combinations of $w_{r-a+1}, \cdots, w_{r+1}$ are the only $1$-degree homogeneous  polynomials satisfying \eqref{relation11}.
	
	For the induction step, let $P (w_1, w_2, \ldots ,w_r, w_{r+1}) $ be a degree $k$-homogenous polynomial  satisfying \eqref{relation11}. 
	
	Since the action of $\rho(\Lambda)$ does not involve the variables $w_1, \ldots, w_r$, we can use the definition of the partial derivative to see that the
	degree $(k-1)$-homogenous polynomials
	$$
	\frac{\partial P}{\partial w_1 } ,\ldots, \frac{\partial P}{\partial w_r} 
	$$
	also satisfy \eqref{relation11}.
	By induction, 	for every $i \leq r$, we obtain that 
	$\frac{\partial P}{\partial w_i} (z)$ depends only on $w_{r-a+1}, \cdots, w_{r+1}$.

	As a consequence, we have
	$$
	P (w) = Q (w_{r-a+1}, \cdots, w_{r+1}) +\sum_{i \leq r-a} w_i P_i (w_{r-a+1}, \cdots, w_{r+1})  ,
	$$
	where the $P_i (w_{r-a+1}, \cdots, w_{r+1})$ and $Q (w_{r-a+1}, \cdots, w_{r+1})$ are  polynomials of degree $k-1$ depending only on $w_{r-a+1}, \cdots, w_{r+1}$.  Together with \eqref{relation11}, we obtain
	$$\sum_{1\leq i \leq r-a } w_i P_i (w_{r-a+1}, \cdots, w_{r+1}) = \sum_{1\leq i \leq r-a } (w_i+ b_i (\Lambda) w_{r+1}) \cdot P_i (w_{r-a+1}, \cdots, w_{r+1}) .$$
	Therefore
	\begin{equation}\label{relatt}
		\sum_{1\leq i \leq r-a }  b_i (\Lambda) \cdot P_i (w_{r-a + 1}, \cdots, w_{r+1}) \equiv 0 .
		\end{equation}
By Lemma \ref{newbase} the vectors
	$$\{(b_1 (\Lambda), \cdots , b_{r-a} (\Lambda)) \}_{\Lambda\in \pi_1 (X)} \subset \mathbb C^{r-a}$$
	generate the linear space $\mathbb C^{r-a}$. Then \eqref{relatt} implies that $P_i (w_{r-a+1}, \cdots, w_{r+1}) \equiv 0$. Therefore $P (w)= Q (w_{r-a+1}, \cdots, w_{r+1}) $, which depends only on $w_{r-a+1}, \cdots, w_{r+1}$. 
\end{proof}

As a corollary of Lemma \ref{singlocus} and Lemma \ref{uniquesects}, we obtain

\begin{corollary}\label{uniquesect}
	Let $X$ be a compact K\"ahler manifold and let $V$ be an extension
$$
0\rightarrow \mathcal{O}_X \rightarrow V \rightarrow \mathcal{O}_X ^{\oplus r} \rightarrow 0 
$$
such that that $h^0 (X, V) =1$. 
Then $h^0 (X, S^k V) =1$ for every $k \in \N$. In particular if 
$D \subset \PP(V)$ is an effective divisor such that $D \in |\sO_{\PP(V)}(k)|$,
then $D=k E$ where $E$ is the divisor induced by the canonical section of $V$.
\end{corollary}

We can now prove the key step for Theorem \ref{genDPS}:

\begin{proposition}\label{exvec}
	Let $X$ be a compact K\"ahler manifold and let $V$ be an extension
	\begin{equation}\label{extens}
		0\rightarrow \mathcal{O}_X \rightarrow V \rightarrow \mathcal{O}_X ^{\oplus r} \rightarrow 0 
		\end{equation}
such that $\dim H^0 (X, V) =1$. 
		Let $h$ be a possibly singular metric on $\mathcal{O}_{\mathbb P (V)} (1)$ such $i\Theta_h (\mathcal{O}_{\mathbb P (V)} (1)) \geq 0$. 
	Then $[\frac{i}{2\pi}\Theta_h (\mathcal{O}_{\mathbb P (V)} (1))] = [E]$, where $E$ is the divisor induced by the canonical section of $H^0 (X, V)$.
\end{proposition}

\begin{proof}
We consider the fibration $\pi: \mathbb P (E) \to X$ and Siu's decomposition of the current
	$$\frac{i}{2\pi}\Theta_h (\mathcal{O}_{\mathbb P (V)} (1)) = \sum a_i [E_i] + \sum b_i [E' _i] +\Theta_0 ,$$
	where $\pi (E_i) =X$ and $\pi (E' _i) \subsetneq X$.
	By Lemma \ref{singlocus}, we know that 
	$$c_1 ([E_i]) = \lambda_i \cdot c_1 (\mathcal{O}_{\mathbb P (V)} (1)) $$
	for some constant $\lambda_i$ and $[E_i]$ is induced by a flat section. Thanks to Lemma \ref{uniquesect}, we know that $[E_i]= [E]$.
	Therefore 
	$$\frac{i}{2\pi}\Theta_h (\mathcal{O}_{\mathbb P (V)} (1)) = a [E] + \sum b_i [E' _i] +\Theta_0 ,$$
	
	\medskip
	
	If  $a=1$, then $\frac{i}{2\pi}\Theta_h (\mathcal{O}_{\mathbb P (V)} (1)) = [E] $ and the lemma is proved. 
	
	If $a<1$, then $\sum b_i [E' _i] +\Theta_0 $ is not zero,
	and we can find a metric $h_1$ on $\mathcal{O}_{\mathbb P (V)} (1)$ such that 
	$$
	\frac{i}{2\pi}\Theta_{h_1}  (\mathcal{O}_{\mathbb P (V)} (1)) = \frac{1}{1- a} \cdot (\sum b_i [E' _i] +\Theta_0 ) \qquad\text{ on } \mathbb P (V). 
	$$
	
	{\em 1st case. Assume that $r=1$.}
	Let $F$ be a generic fiber of $\pi$.
	Since $\dim F=1$ and the non zero Lelong number locus of $\Theta_0$ is of codimension at least $2$, 
	the Lelong number of $i\Theta_{h_1}  (\mathcal{O}_{\mathbb P (V)} (1)) $ on a very general fiber $F$ vanishes. 
	
	\medskip
	
	We now equip $\mathcal{O}_{\mathbb P (V)} (3)$ with the metric $h_1 ^{\otimes 3}$, and it induces a $L^2$-metric $h_{L^2}$ on 
	$$
	\pi_* (K_{\mathbb P (V) /X} \otimes \mathcal{O}_{\mathbb P (V)} (3)) = V \otimes \det V \simeq V. 
	$$
	Here we use the key point that  the restriction of the Lelong number of $i\Theta_{h_1}  (\mathcal{O}_{\mathbb P (V)} (1)) $ on a generic fiber $F$ vanishes. So we don't need to twist with the multiplier ideal sheaf in the direct image. 
	Thanks to \cite{PT18}, we know that 
	$$i\Theta_{h_{L^2}} (V \otimes \det V ) \succeq 0$$ 
	in the sense of Griffiths (singular version cf. for example \cite[Def 2.9]{Pau16}). Since $c_1(V)=0$ we can use \cite[Cor 2.9 (b)]{CP17} to obtain that $V$ is unitary flat and we obtain a contradiction with the fact that $V$ is a non trivial extension.	

{\em 2nd case. Assume that $r>1$.}
We argue by induction on $r$, the case $r=1$ being done above.

Choose a very general quotient line bundle $\mathcal{O}_X  ^{\oplus r} \rightarrow \sO_X$, and let $V_1 \subset V$ be the kernel of 
$V \rightarrow \mathcal{O}_X  ^{\oplus r} \rightarrow \sO_X.$
Then we have exact sequences
\begin{equation}\label{extt1}
	0 \rightarrow \mathcal{O}_X \rightarrow V_1 \rightarrow \mathcal{O}_X  ^{\oplus (r -1)} \rightarrow  0
	\end{equation} 
and 
\begin{equation}\label{extt}
	0 \rightarrow V_1 \rightarrow V \rightarrow \mathcal{O}_X \rightarrow 0,
\end{equation}
in particular $1 \leq h^0(X, V_1) \leq h^0(X, V) \leq 1$, so $V_1$ satisfies the induction
hypothesis. Let now $\mathbb P (V) \dashrightarrow \mathbb P (V_1)$ be the projection
from $\PP(\sO_X) \subset \PP(V)$ onto $\PP(V_1)$.
Following the notation of Remark \ref{remark-projection} we have a commutative diagram

$$
\xymatrix{
	\widehat{\mathbb P} (V) \ar[d]^\mu \ar[rd]^q & \\
	\mathbb P (V) \ar@{.>}[r] \ar[d] & \mathbb P (V_1) \ar[dl]\\
 X &  }
 $$
We claim that the Lelong number of the restriction of $\mu^*	i\Theta_{h_1} (\mathcal{O}_{\mathbb P (V)} (1)) $ on the very general fiber of $q$ vanishes. 

{\em Proof of the claim.}
Since $E_i'$ does not surject onto $X$, it is sufficient to show that
 the Lelong number of $\mu^* \Theta_0$ on the very general fiber of $q$ vanishes.  The singular locus $Z$ of $\Theta_0$ has irreducible components of codimension at least two, in particular $Z \cap \PP(\sO_X^{\oplus r})$
 is a proper subset of the divisor $\PP(\sO_X^{\oplus r}) \subset \PP(V)$.
 Since we have chosen a very general quotient line bundle $\mathcal{O}_X  ^{\oplus r} \rightarrow \sO_X$, the variety $\PP(\sO_X) \subset \PP(V)$ is not contained
 in $Z$. Choose now $x \in X$ a very general point such that 
 $$
 F := \fibre{\pi}{x} \simeq \mathbb P^{r}
 $$
 has the property that all irreducible components of $Z \cap F$ have codimension at least
 two and the point $y :=\PP(\sO_X) \cap F$ is not in $Z$.
The family of lines passing through the point $y$ form a free family, so by \cite[II, Prop.3.7]{Kol95} a very general line through $y$ is disjoint from $Z$. Since
the fibres of $q$ identify to the lines through $\PP(\sO_X)$, this shows the claim.

\medskip

Now we equip $\mu^* \mathcal{O}_{\mathbb P (V)} (3)$ with the metric $\mu^* h_1 ^{\otimes 3}$. Then $h_1 ^{\otimes 3}$ induces a $L^2$-metric $h_{L^2}$ on 
$q_* (K_{	\widehat{\mathbb P} (V) /  \mathbb P (V_1) } \otimes \mu^* \mathcal{O}_{\mathbb P (V)} (3))$ and we have
$$i\Theta_{h_{L^2}} (q_* (K_{	\widehat{\mathbb P} (V) /  \mathbb P (V_1) } \otimes \mu^* \mathcal{O}_{\mathbb P (V)} (3)))  \succeq 0  \qquad\text{ on }  \mathbb P (V_1) .$$
Here we use the fact that the Lelong number of the restriction of $\mu^*	i\Theta_{h_1} (\mathcal{O}_{\mathbb P (V)} (1)) $ on the generic fiber of $q$ vanishes. So we don't need to twist with the multiplier ideal sheaf with respect to the metric $\mu^* h_1 ^{\otimes 3}$.
By taking the determinant, we have
$$
i\Theta_{\det h_{L^2}} (\det q_* (K_{	\widehat{\mathbb P} (V) /  \mathbb P (V_1) } \otimes \mu^* \mathcal{O}_{\mathbb P (V)} (3)))  \geq 0  \qquad\text{ on }  \mathbb P (V_1) $$
in the sense of currents.
\medskip

From \eqref{extension-two} we compute that we have an extension
\begin{equation}\label{extennonsplit}
	0 \to \sO_{\PP(V_1)}(2) \to q_* (K_{	\widehat{\mathbb P} (V) /  \mathbb P (F) } \otimes \mu^* \mathcal{O}_{\mathbb P (V)} (3)) \rightarrow \sO_{\PP(V_1)}(1)\to 0. 
	\end{equation}
Since this extension is a twist of \eqref{extension-two}, and the extension \eqref{extt1} is not split,
we know by Remark \ref{remark-projection} that \eqref{extennonsplit} does not split.

\medskip

Since 
$$
\det q_* (K_{	\widehat{\mathbb P} (V) /  \mathbb P (V_1) } \otimes \mu^* \mathcal{O}_{\mathbb P (V)} (3))
\simeq
\mathcal{O}_{\mathbb P (V_1)}(3)
$$
and $V_1$ satisfies the induction hypothesis we know that
$$
\frac{i}{2\pi}\Theta_{\det h_{L^2}} (\det q_* (K_{	\widehat{\mathbb P} (V) /  \mathbb P (V_1) } \otimes \mu^* \mathcal{O}_{\mathbb P (V)} (3))) = 3 [E_1] ,
$$
where $E_1 \subset \PP (V_1)$ is the divisor induced by the canonical section
of $V_1$.

\medskip

Now we apply \cite[Cor 2.9 (b)]{CP17} to $i\Theta_{h_{L^2}} (q_* (K_{	\widehat{\mathbb P} (V) /  \mathbb P (V_1) } \otimes \mu^* \mathcal{O}_{\mathbb P (V)} (3)))$ over $\mathbb P (V_1)  \setminus E_1$, we know that 
\begin{equation}\label{herm}
	i\Theta_{h_{L^2}} (q_* (K_{	\widehat{\mathbb P} (V) /  \mathbb P (V_1) } \otimes \mu^* \mathcal{O}_{\mathbb P (V)} (3))) \equiv 0 \qquad\text{ on } \mathbb P (V_1)  \setminus E_1 .
	\end{equation}
Therefore the extension \eqref{extennonsplit} splits on $\mathbb P (V_1)  \setminus E_1 $, namely there is an injective holomorphic morphism 
$$\varphi: \mathcal{O}_{\mathbb P (V_1)} (b) \rightarrow q_* (K_{	\widehat{\mathbb P} (V) /  \mathbb P (V_1) } \otimes \mu^* \mathcal{O}_{\mathbb P (V)} (3)) \qquad\text{on } \mathbb P (V_1)  \setminus E_1 $$ 
which preserves also the metric. 
Note that the restriction of the extension \eqref{extennonsplit} on a generic fiber of $\mathbb P (V_1) \rightarrow X$ splits and the splitting preserves the metric.  Therefore $\varphi$ extends to some open subset of $\mathbb P (V_1)$ of codimension at least $2$. By Hartogs, $\varphi$ is well defined on the total space $\mathbb P (V_1)$. In particular, \eqref{extennonsplit} splits and we get a contradiction.
\end{proof}	

Proposition \ref{exvec} essentially settles the case $a=0$ of Theorem \ref{genDPS}. For the case $a >0$, we need another lemma.

\begin{lemma}\label{flatcuv}
	Let $X$ be a compact K\"ahler manifold and let $V$ be an extension
$$
0\rightarrow \mathcal{O}_X ^{\oplus r_1}  \rightarrow V \rightarrow \mathcal{O}_X  ^{\oplus r_2} \rightarrow 0
$$
such that $\dim H^0(X, V)=r_1$.
Let $h$ be a possibly singular metric  on $\mathcal{O}_{\mathbb P (V)}(1)$ such that $i\Theta_h (\mathcal{O}_{\mathbb P (V)}(1)) \geq 0$. 
Set $\varphi:= \log h^*$, which is a function defined on $V^*$:
$$\varphi : V^* \rightarrow \mathbb R \cup \{-\infty\} .$$

Let $\sigma: \widetilde{X} \rightarrow X$ be the universal cover and we consider the diagram as in \eqref{comdia}
\begin{equation}\label{comdia11}
	\begin{CD}
	\sigma^* V^* \simeq \widetilde{X}  \times \mathbb C^{r_1+r_2}@>{\sigma_{V^*}}>> V^*\\
		@V \tilde \pi VV      @VV\pi V  \\ 
		\widetilde{X} @>>{\sigma}>   X
	\end{CD}	
\end{equation}
Then there exists psh function $\psi$ on $\mathbb C^{r_1 +r_2}$ such that $\sigma_{V^*} ^* \varphi =  \psi \circ pr$, where $pr$ is the natural map $\sigma^* V^* \simeq \widetilde{X}  \times \mathbb C^{r_1+r_2} \rightarrow \mathbb C^{r_1 +r_2}$.
\end{lemma}	

\begin{proof}
	We prove the theorem by induction on $r:=r_1 +r_2$.  
	
	If $r_1 +r_2 =2$,  thanks to Proposition \ref{exvec}, we know that  $\frac{i}{2\pi}\Theta_h (\mathcal{O}_{\mathbb P(V)} (1)) =[E]$, where $[E]$ is the divisor corresponding to $\mathbb P (\mathcal{O}_X ^{\oplus r_2} ) \subset \mathbb P (V)$.  Then $\sigma_{V^*} ^*\varphi = \psi \circ pr$, where $\psi = \log |z_1|^2 +C$ is a psh function on $\mathbb C^2$ for some constant $C$.
	
	\medskip
	
	Now we suppose by induction that the theorem holds for $r_1 +r_2 =r-1$. 
	If $r_1 =1$, the statement is proved in Proposition \ref{exvec} together with the above argument. From now on, we suppose that $r_1 \geq 2$.
	
	Let $(w_1, w_2) \in  \mathbb{C}^{r_1}\times  \mathbb{C}^{r_2}$ be the standard coordinates, and let $z \in \widetilde X$ be any point. 	
	Since $V^*$ is an extension
	$$0\rightarrow (\mathcal{O}_X ^{\oplus r_2})^* \rightarrow V^* \rightarrow (\mathcal{O}_X ^{\oplus r_1})^* \rightarrow 0 , $$
	the action of $\pi_1 (X)$ on $\sigma^* V^*$ is given by
	\begin{equation}\label{relation1}
		(z, w_1, w_2 ) \to ( \Psi (\Lambda, z) ,  w_1 ,  w_2 + M (\Lambda) \cdot  w_1)  \qquad\text{for } \Lambda\in \pi_1 (X), w_1 \in \mathbb{C}^{r_1}, w_2 \in \mathbb{C}^{r_2} ,
	\end{equation}
	where $\Psi (\Lambda, z)$ is the natural action $\Psi : \pi_1 (X) \times \widetilde{X} \rightarrow \widetilde{X}$, and $M (\Lambda)  $ is a $r_2 \times r_1$-matrix induced by the local system.
	
	\medskip
	
	Now we consider the psh function $\sigma_{V^*}^*\varphi$ on $ \sigma^* V^* \simeq\widetilde{X}  \times   \mathbb{C}^{r_1 }\times  \mathbb{C}^{r_2}$. 
	We take a general non zero section $s\in H^0 (X, \mathcal{O}_X ^{\oplus r_1})$. It corresponds to an injective morphism 
	\begin{equation}\label{injectind}
		\mathcal{O}_X \rightarrow \mathcal{O}_X ^{\oplus r_1}.
	\end{equation} 
	By composing it with $\mathcal{O}_X ^{\oplus r_1} \rightarrow V$, we obtain a filtration
	$$0 \rightarrow \mathcal{O}_X \rightarrow V \rightarrow V_s  := V/\mathcal{O}_X \rightarrow 0.$$
	By passing to duality, we know that $V_s ^* \subset V^*$ is a subbundle.  	Now for a generic section $s\in H^0 (X, \mathcal{O}_X ^{\oplus r_1})$, the restriction $\varphi |_{V_s ^*}$ is not identically $-\infty$. 
	
	We have $\dim H^0 (X, V_s ) = r_1 -1+ b$ for a certain $b \geq 0$ and hence an extension
	$$0 \rightarrow \mathcal{O}_X ^{\oplus r_1 -1 +b} \rightarrow V_s  \rightarrow \mathcal{O}_X ^{\oplus r_2 +1-b} \rightarrow 0 
	$$
	such that $h^0 (X, V_s ) = r_1 -1+ b$.
	By applying the induction hypothesis to $V_s$, 
	we know that $(\sigma_{V^*} ^*\varphi) |_{ \sigma^* V_s ^*}$ is constant with respect to $z \in \widetilde{X} $.  Since 
	the subbundles $\sigma^* V_s^*$
	cover a Zariski open subset of $\sigma^* V^*$, we know that $\sigma_{V^*} ^*\varphi $ is constant with respect to $z$. 
\end{proof}	

Now we are ready to prove the Theorem \ref{genDPS}.

\begin{proof}[Proof of Theorem \ref{genDPS}]

Fix a point $x\in X$. Thanks to Lemma \ref{newbase}, we can choose a base of $V^* _x$ such that the representation $\rho : \pi_1 (X) \to GL (r+1, \mathbb C)$ is of the form 
\[
\rho (\Lambda) = \begin{bmatrix}
	1 & 0 & 0 & \cdots& b_{1} (\Lambda)\\
	0 & 1 & 0 & \cdots & b_2 (\Lambda)\\
	\vdots & 0 & 1 & \cdots & b_{r-a} (\Lambda)\\
	\vdots & \vdots &  0 & \ddots & 0 \\
	\vdots & \vdots &  0 & \ddots & \cdots  \\
	0 & \cdots & \cdots & \cdots & 1 
\end{bmatrix}  \qquad\text{ for } \Lambda\in \pi_1 (X).
\]

Thanks to Lemma \ref{flatcuv}, we know that $\sigma_{V^*}^*\varphi = \psi \circ pr$ for some function $\psi$ on $\mathbb C^{r+1}$. Since $\psi \circ pr$ is the pull-back of some function on $V^*$, we obtain
	\begin{equation}\label{relation2}
		\psi( w_1, \cdots , w_{r+1} ) =\psi ( w_1 + b_1 (\Lambda) w_{r+1},  \cdots, w_{r-a} + b_{r-a} (\Lambda) w_{r+1},  w_{r-a+1} ,\cdots,  w_{r+1})
	\end{equation}
for any  $\Lambda \in \pi_1 (X)$.
	\smallskip

Fix a $\Lambda \in \pi_1 (X)$ and a point $(w_1, \cdots, w_{r+1}) \in \mathbb C^{r+1}$ such that $w_{r+1} \neq 0$ and $\psi (w_1, \cdots, w_{r+1})$ is finite.
We consider the function $S $ on $\mathbb C $ as follows
\begin{equation}\label{1per}
	S(\lambda) := \psi ( w_1 + \lambda \cdot b_1 (\Lambda),  \cdots, w_{r-a} + \lambda\cdot b_{r-a} (\Lambda) ,  w_{r-a+1} ,\cdots,  w_{r+1}) \text{ for } \lambda\in \mathbb C.
\end{equation}
Thanks to \eqref{relation2}, we know that $S(\lambda)= S (\lambda +w_{r+1})$. 
Then $S$
is a $1$-periodic psh function on $\mathbb C$ with logarithmic growth at infinity. 
As explained in the proof of \cite[Example 1.7]{DPS94}, $S$ is  a constant function.
\medskip

Now we fix some numbers $w_{r-a+1} ,\cdots,  w_{r+1} $ such that $w_{r+1} \neq 0$.
We consider the following function on $\mathbb C^{r-a} $: 
\begin{equation}\label{bounpsh}(w_1, \cdots, w_{r-a}) \in \mathbb C^{r-a}  \to \psi (w_1, \cdots, w_{r-a}, w_{r-a+1} ,\cdots,  w_{r+1} )
	\end{equation}
Since $S (\lambda)$ in \eqref{1per} is constant and 
$$
\{(b_1 (\Lambda), \cdots, b_{r-a} (\Lambda) \}_{\Lambda\in \pi_1 (X) }
$$ 
generates $\mathbb{C}^{r-a}$ as a $\C$-vector space by Lemma \ref{newbase}, the function \eqref{bounpsh} is thus constant. In other words, $\psi$ is constant with respect to $w_1, \cdots, w_{r-a}$.  Therefore $\varphi$ comes from a function on $\mathbb C^{1+a}$. 
\end{proof}	

It will be very interesting to prove Theorem \ref{genDPS} in a more general setting, for example when $V$ is an extension of two hermitian flat vector bundles.  However, the following example tell us that the situation will be rather complicated:

\begin{example} \label{example-bad-case}
Let $A$ be an elliptic curve and let $V_A$ be the Serre vector bundle, cf. Example \ref{example-DPS}. Let $V:= V_A^{\oplus 2}$, so 
$V$ is an extension of 
$$
0 \rightarrow \mathcal{O}_X ^{\oplus 2} \rightarrow V  \rightarrow \mathcal{O}_X ^{\oplus 2} \rightarrow 0
$$
such that $h^0(X, V)=h^0(X, \mathcal{O}_X ^{\oplus 2})=2$.
By \cite[Lemma 17, 1)]{Ati57} we have $h^0(X, V_A \otimes V_A)=2$, so 
$$
\C^4 \simeq H^0(X, S^2 V) \not\simeq S^2 H^0(X, V) \simeq \C^3. 
$$ 
Thus even the algebraic Lemma \ref{uniquesects} does not generalise to this setting.
More explicitly the representation
$$\rho : \pi_1 (X) \rightarrow \mbox{GL} (4, \mathbb C)$$
is given by matrices of the form
\[
\rho (\Lambda) = \begin{bmatrix}
	1 & b_{1} (\Lambda) & 0 & 0 \\
	0 & 1 & 0 & 0 \\
	0 & 0 & 1 & b_{1} (\Lambda)\\
	0 &  0 & 0 & 1 
\end{bmatrix}  \qquad\text{ for } \Lambda\in \pi_1 (X).
\]
In these coordinates $w_1^2, w_3 ^2, w_1 w_3$ are the obvious sections of $H^0 (X, S^2 V)$, but
$$ w_1 w_4 -w_2 w_3$$
is also invariant with respect to $\rho$-action, hence it defines also a flat section of $H^0 (X, S^2 V)$.
In particular, $\ln |w_3 w_2 -w_1 w_4|$ defines a singular metric on $\mathcal{O}_{\mathbb P (V)} (1)$ which depends on all four variables.
\end{example}	

\begin{proof}[Proof of Theorem \ref{theorem-direct-image}]
The restriction of the natural morphism $f^* f_* \Omega_X \rightarrow \Omega_X$
to a general fibre $A$ identifies to the evaluation map
$$
H^0(A, \Omega_X \otimes \sO_A) \otimes \sO_A \rightarrow \Omega_X \otimes \sO_A.
$$
Since $\Omega_X \otimes \sO_A$ is a numerically flat vector bundle, the image of the evaluation map is isomorphic to the trivial bundle $H^0(A, \Omega_X \otimes \sO_A) \otimes \sO_A$ (combine \cite[Lemma 1.20]{DPS94} with \cite[Prop.1.16]{DPS94}). In particular $f^* f_* \Omega_X \rightarrow \Omega_X$
is a saturated subbundle of $\Omega_X$ in a neighbourhood of a general
fibre, hence the restriction of its saturation $\sV \subset \Omega_X$ to $A$ is a trivial vector bundle.

The inclusion $\sV \rightarrow \Omega_X$ is injective as a morphism of vector bundle
in the complement of a subset $Z \subset X$ that has codimension at least two.
There the dual map
	$$
	 i : \Omega_X ^* \rightarrow \sV^*
	 $$
	 is surjective in the sense of vector bundles on $X \setminus Z$.

	Since $\Omega_X$ is pseudo-effective, there exists a possibly singular metric $h$ on $\mathcal{O}_{\PP (\Omega_X)} (1)$ such that 
	$i\Theta_h (\mathcal{O}_{\PP (\Omega_X)} (1)) \geq 0$. By the Poincaré-Lelong formula, $\log h^* : \Omega_X ^* \rightarrow \mathbb R$ is a psh function.
	We would like to prove that $\log h^* |_{\Omega_X ^* |_{X \setminus Z}}$ is the pull back of some psh function $\psi$ on $\sV^* |_{X \setminus Z}$. If it is proved, $e^{\psi}$ induces by duality a metric on $\mathcal{O}_{\PP (\sV |_{X\setminus Z})} (1)$ of semipositive curvature current. Then 
	$\sV |_{X\setminus Z}$ is pseudo-effective.
	
	Let $A$ be a generic fiber of $f: X\rightarrow Y$. Then 
	$$
	0 \rightarrow \sV \otimes \sO_A \rightarrow \Omega_X \otimes \sO_A \rightarrow Q \otimes \sO_A \rightarrow 0,
	$$
	is an extension over the torus $A$ and we have $\rk \sV \otimes \sO_A =  h^0 (A, \sV \otimes \sO_A) = h^0 (A, \Omega_X \otimes \sO_A)$.  
	Then we can apply Theorem \ref{genDPS}: namely $\log h^* |_{\Omega^* _X \otimes \sO_A  }$ comes from a psh function on $\sV^* \otimes \sO_A$. 
	As it holds for a generic fiber $A$, we can thus find a Zariski open set $U \subset X$ such that 
	$\log h^* |_{\Omega^* _X |_U }$ is the pull back of some function $\psi$ on $\sV^* |_U$, i.e.,
	$$\log h^*  = i^* \psi \qquad\text{on } \Omega^* _X |_U, $$
	where $i$ is the natural morphism $\Omega^* _X \rightarrow \sV^*$.
	In particular, $\psi $ is psh on $\sV^* |_U$. Let us show that $\psi$ extends as a psh function to the space $\sV^* |_{X\setminus Z}$:
	since $i$ is a the surjective morphism in the sense of vector bundle over $X\setminus Z$, the locally upper boundedness of $\log h^* $ implies that $\psi$ is locally upper bounded near a generic point of $(X\setminus Z) \setminus U$. Therefore $\psi$ can be extended as a psh function on $\sV^* |_{X\setminus Z}$. It induces a possibly singular metric $h_\sV$ on $\mathcal{O}_{\mathbb P (\sV) } (1)$ over $X\setminus Z$ such that
	$$i\Theta_{h_\sV} (\mathcal{O}_{\mathbb P (\sV) } (1)) \geq 0 \qquad \text{over } X\setminus Z .$$
	Let $r: P \to \mathbb P ' (\sE)$ and $\nu: \mathbb P ' (\sE) \to \mathbb P (\sE)$ be the resolutions in Definition \ref{definitiontautological}. As the morphism $\sV \rightarrow \Omega_X $ is algebraic, we know that $(r\circ \nu)^* h_\sV$ can be extended as a possibly singular metric $h_1$ on $\mathcal{O}_P (1)$ such that  
	$$i\Theta_{h_1} (\mathcal{O}_P (1)) + C_1 [D]\geq 0\qquad\text{ on } P $$
	for some constant $C_1$.  Then $\sV$ is pseudo-effective by Remark \ref{anav}.
	\end{proof}

While the subsheaf $\sV \subset \Omega_X$ captures the pseudoeffectivity of $\Omega_X$, we should not expect
the quotient $\Omega_X/\sV$ to be ``negative'':

\begin{example} \label{example-nef}
Let $\holom{f}{X}{Y}$ be a smooth non-isotrivial fibration
over a curve $Y$ such that all the fibres are abelian varieties.
Then the relative cotangent sheaf $\Omega_{X/Y}$ is nef. 
If $g(Y) \geq 1$ (which we can assume after base change), the cotangent sheaf $\Omega_X$ is pseudoeffective.
Since $f$ is not isotrivial,
we have $\sV \subsetneq \Omega_X$.
Thus  the quotient sheaf
$\Omega_X/\sV$ has rank at least one and is a quotient of $\Omega_{X/Y}$. Since $\Omega_{X/Y}$ is nef, its quotient $\Omega_X/\sV$ is also nef.
\end{example}

\section{Two nonvanishing results}

In this final section we provide some evidence towards Conjecture \ref{conjecture-dichotomy}, in particular we prove that the expected
dichotomy \ref{conjecture-dichotomy} actually implies the nonvanishing conjecture for fibrations over curves.

\subsection{Fundamental group}

\begin{lemma} \label{lemmafundamentalgroup}
Let $X$ be a projective manifold that admits a fibration $\holom{f}{X}{C}$
onto a curve $C$ such that the general fibre $F$ has $c_1(F)=0$. Then $\tilde q(X)>0$ or $X$ has finite fundamental group.
\end{lemma}

\begin{proof}
We can assume without loss of generality that $C \simeq \PP^1$, since otherwise $q(X)>0$.
Denote by $(C, \Delta= \sum (1-\frac{1}{m_i}) p_i$ the orbifold structure on the base $C$ \cite[Defn.3.5]{CC14}.

If $K_C+\Delta$ has positive degree we can copy the proof of \cite[Prop.5.2]{HP20} to show that
there exists a ramified covering $C' \rightarrow C$ that branches exactly with
order $m_i$ in $p_i$ and is \'etale elsewhere. Thus the normalisation $X'$
of the fibre product $X \times_C C'$ has an \'etale map to $X$.
By the Hurwitz formula we have $g(C')>0$, so we obtain $q(X')>0$. 

If $K_C+\Delta$ has non-positive degree the fundamental group $\pi_1(C, \Delta)$
is almost abelian \cite[Ex.4.1]{CC14}. Since $c_1(F)=0$, the Beauville-Bogomolov decomposition theorem implies that the fundamental group $\pi_1(F)$ is also almost abelian. Thus by \cite[Thm.4.1]{CC14} the fundamental group $\pi_1(X)$ is almost abelian. If $\pi_1(X)$ is not finite there exists thus an \'etale cover $X' \rightarrow X$ such that $q(X') \neq 0$.
\end{proof}

A vector bundle $V$ on a projective manifold is said to be 
projectively Hermitian flat if it admits a smooth Hermitian metric $h$ such that its Chern curvature tensor $R = \nabla^2$ can be written as $R = \alpha \cdot \id_V$ for some $2$-form $\alpha$. In particular, the associated projectivised bundle $\PP(V)$ is given by a representation $\pi_1(X) \rightarrow PU(\rk V)$.
In our situation it is convenient to start with a more flexible notion:

\begin{definition} \label{definition-projectively-flat} \cite[Defn.4.1]{LOY20}
Let $X$ be a projective manifold. Let $V$ be a vector bundle on $X$,
denote by $\holom{\pi}{\PP(V)}{X}$ the projectivisation and by $\zeta_V$ the tautological class on $\PP(V)$. We say that $V$ is numerically projectively flat if the divisor class $\zeta_V - \frac{1}{\rk V} \pi^* c_1(V)$ is nef.
\end{definition}

The relation between the two notions is clarified by the following:

\begin{theorem} \cite[IV,Thm.4.1]{Nak04} \label{theorem-filtrate}
Let $X$ be a projective manifold, and let $V$ be a numerically projectively flat vector bundle on $X$. Then there exists a filtration of subbundles
$$
0 = V_0 \subsetneq V_1 \subsetneq \ldots \subsetneq V_p=V
$$
such that the graded pieces $V_i/V_{i-1}$ are projectively Hermitian flat and
$\frac{1}{\rk V_i/V_{i-1}} c_1(V_i/V_{i-1})$ is numerically equivalent to
$\frac{1}{\rk V} c_1(V)$.
\end{theorem}

\begin{proposition} \label{proposition-projectively-flat}
Let $X$ be a projective manifold with $\tilde q(X)=0$ that admits a fibration $\holom{f}{X}{C}$
onto a curve $C$ such that the general fibre $F$ has $c_1(F)=0$.
Let $V \rightarrow X$ be a projectively flat vector bundle of rank $r$ such that $V \otimes \sO_F \simeq \sO_F^{\oplus \rk V}$. If $V$ is pseudoeffective we have $\kappa(X, V) \geq 0$. 

More precisely there exists a finite \'etale cover $\holom{\nu}{X'}{X}$ such that
$\nu^* V \simeq L^{\oplus r}$ where $L$ is a line bundle with $\kappa(X,L) \geq 0$.
\end{proposition}

\begin{proof}
By Lemma \ref{lemmafundamentalgroup} the fundamental group $\pi_1(X)$ is finite. Since numerical projective flatness \cite[Lemma 4.3.(3)]{LOY20} and the Kodaira dimension  \cite[Lemma 3.15]{HP20} are invariant under 
finite covers, we can assume without loss of generality that 
$X$ is simply connected. 

Let 
$$
0 = V_0 \subsetneq V_1 \subsetneq \ldots \subsetneq V_p=V
$$
be the filtration from Theorem \ref{theorem-filtrate}.
Since $\pi_1(X)$ is trivial, the projectively Hermitian flat vector bundles
$V_i/V_{i-1}$ are a direct sum of isomorphic line bundles $L_i$. 
By Theorem \ref{theorem-filtrate} 
$\frac{1}{\rk V_i/V_{i-1}} c_1(V_i/V_{i-1}) \simeq L_i$ is numerically equivalent to
$\frac{1}{\rk V} c_1(V)$.
Since $\pi_1(X)$ is trivial, the numerically equivalent line bundles are actually isomorphic, so the line bundle $L_i$ does not depend $i$.
Since $H^1(X, \sO_X)=0$ we deduce that the extensions defined by the filtration split, so we obtain that
$$
V \simeq L^{\oplus r}
$$
where $L$ is a line bundle on $X$. Since $V$ is pseudoeffective and a direct sum it is not difficult to deduce via Definition \ref{def:reflexive}
that the line bundle $L$ is pseudoeffective. Moreover, since
$V \otimes \sO_F \simeq \sO_F^{\oplus \rk V}$, we also have $L \otimes \sO_F \simeq \sO_F$.

We will now show that $\kappa(X, L) \geq 0$: let $L=P+N$ be the divisorial Zariski division of $L$. Since $c_1(L)|_F \equiv 0$ the negative part $N$ is contained in fibres of $f$, so $P|_F \equiv 0$. 
Let $A$ be any ample class on $X$, then we have 
$P^2 \cdot A^{n-2} \geq 0$ (a consequence of \cite[Prop.2.4]{Bou04}).
Thus we can apply Lemma \ref{lemmanumericalclassfibres}
to see that $P \equiv \lambda F$ for some $\lambda \in \R$. 
Since $X$ is simply connected, we actually have a linear equivalence
$P \sim_\R \lambda F$.
The class $P$ being pseudoeffective we have $\lambda \geq 0$ and hence
$\kappa(X, P) \geq 0$. Since $\kappa(X, L) = \kappa(X, P)$ this concludes the proof.
\end{proof}

\subsection{An extension argument}

\begin{proposition} \label{proposition-base-positivity}
Let $\holom{f}{X}{Y}$ be a fibration between projective manifolds. Let $\sV$ be a pseudoeffective reflexive sheaf on $X$ such that 
$\sV \otimes \sO_{F} \simeq \sO_{F}^{\oplus \rk V}$   for a general fibre $F$ and we suppose that $\kappa (F) \geq 0$. 
Let $a\in \mathbb N^*$ such that $H^0 (F, a K_F) \neq 0$.

Then there exists an ample line bundle $A_Y$ on $Y$ such that the image of the restriction
\begin{equation}\label{surj}
	H^0 (X, a K_{X/Y} \otimes S^{[m]} \sV  \otimes f^* A_Y) \rightarrow H^0 (F, a K_F \otimes S^{[m]} \sV  \otimes A_Y) \simeq H^0 (F, a K_F \otimes \sO_{F}^{\oplus {m + \rk V-1 \choose \rk V}})
	\end{equation}
is not zero for a generic fiber $F$ and every $m\in \mathbb N^*$.

In particular, if $c_1 (K_{F})=0$, for every $\varepsilon>0$ the $\Q$-twist
$V(\varepsilon \cdot c_1(f^* A_Y))$ has Kodaira dimension at least $\rk \sV-1 +\dim Y$.
\end{proposition}

\begin{proof}
 Let $p: \mathbb P (\sV) \rightarrow X$ be the projection and let $\nu \circ r: P \rightarrow \mathbb P (\sV)$ be the resolution in Definition \ref{definitiontautological}. 
 Set $\pi = p \circ \nu \circ r :  P \rightarrow X$.

Let $\zeta$ be a tautological class on $P$, and set $g:=f \circ \pi$.

Note that there exists a possibly singular metric $h_1$ on $\mathcal O_{P} (\zeta)$ satisfying the following two conditions:

(i)$h_1$ is smooth on the generic fiber of $g$

(ii)$i\Theta_{h_1} (\mathcal O _{P} (\zeta))\geq -g^*\omega_Y$ for some  K\"ahler metric $\omega_Y$ on $Y$.  

Indeed we can find an ample line bundle $A_Y$ on $Y$ such that
$f_* \sV \otimes A_Y$ is globally generated. Since $\sV \otimes \sO_{F}$ is trivial, the natural morphism
$$
f^* (f_* \sV \otimes A_Y) \rightarrow \sV \otimes f^* A_Y
$$
is surjective on the general fibre. Thus $\sV \otimes f^* A_Y$ is globally generated near the general fibre $F$ and we can use the corresponding global sections of 
$\mathcal O _{P} (\zeta) \otimes \pi^*f^* A_Y$ to construct a singular metric satisfying the above two conditions.

\smallskip

By assumption $\sV$ is pseudoeffective, so there exists a possibly singular metric $h_2$
on $\mathcal O _{P} (\zeta)$ such that $\Theta_{h_2} (\mathcal O _{P} (\zeta)) \geq 0$ on $P$. 
Now we consider the line bundle 
$$
 a K_{P/Y} \otimes \mathcal O _{P}(\zeta)^{\otimes (m + r)} \otimes g^* A_Y
$$
We equip $O _{P}(\zeta)^{\otimes (m + r)}$ with the metric $mh_2 + r h_1$.
Let $G$ be a general fibre of $g$. Since $\sV$ is locally free near a general fibre, we have
$$
G \simeq \PP(\sV \otimes \sO_{F}), 
$$
in particular $\mathcal O _{P}(\zeta) \simeq \sO_{\PP(\sV \otimes \sO_{F})}(1)$.

By Ohsawa-Takegoshi extension theorem \cite[Thm.4.1.1]{PT18}, together with the fact that $h_1$ is smooth, we know that 
$$H^0 (P, a K_{P/Y} \otimes \mathcal O _{P}(\zeta)^{\otimes (m + r)} \otimes g^* A_Y ) \rightarrow H^0 (G, 
a K_{P/Y} \otimes \mathcal O _{P}(\zeta)^{\otimes (m + r)} \otimes \mathcal{I} (m h_2 +r h_1)\otimes g^* A_Y
)$$
$$=H^0 (G, 
a K_{P/Y} \otimes \mathcal O _{P}(\zeta)^{\otimes (m + r)} \otimes \mathcal{I} (m h_2 )\otimes g^* A_Y
)$$
is surjective for a general fibre $G$ and every $m\in \mathbb N^*$.
Note that $G =F\times \mathbb P^{r-1}$. we have
\begin{equation}\label{iden}
H^0 (G, 
a K_{P/Y} \otimes \mathcal O _{P}(\zeta)^{\otimes (m +r)}\otimes \mathcal{I} (m h_2)  \otimes g^* A_Y
)
\end{equation}
$$
\simeq
H^0 (F, a K_F ) \otimes H^0 (\mathbb P^{r-1},  \mathcal{O} (m) \otimes  \mathcal{I} (m h_2 )) .
$$
By using \cite[Lemma 4.4]{CCP19}, we know that $H^0 (\mathbb P^{r-1},  \mathcal{O} (m) \otimes  \mathcal{I} (m h_2 )) \neq 0$. 
We obtain thus \eqref{surj}.

\medskip

For the second part of the proposition, by the same argument, we know that the restriction
$$H^0 (P, a K_{P/Y} \otimes \mathcal O _{P}(\zeta)^{\otimes (m + m')} \otimes g^* m' A_Y ) \rightarrow H^0 (G, 
a K_{P/Y} \otimes \mathcal O _{P}(\zeta)^{\otimes (m + m')} \otimes \mathcal{I} (m h_2 +m' h_1)\otimes g^* m' A_Y
)$$
$$=H^0 (G, 
a K_{P/Y} \otimes \mathcal O _{P}(\zeta)^{\otimes (m + m')} \otimes \mathcal{I} (m h_2 )\otimes g^* m'A_Y
)$$
is surjective for every $m, m' \in\mathbb N$. In particular, if we take $m' =m\cdot \ep$, we know that 
$$H^0 (G, 
a K_{P/Y} \otimes \mathcal O _{P}(\zeta)^{\otimes (m + m')} \otimes \mathcal{I} (m h_2 )\otimes g^* m'A_Y
) \sim (\ep m)^{r-1} .$$
The second part is proved.
\end{proof}

\subsection{Towards a dichotomy}

\begin{proposition} \label{proposition-dichotomy}
Let $X$ be a projective manifold such that $q(X)=0$.
Assume that $X$ admits a fibration $\holom{f}{X}{\PP^1}$ with irreducible fibres.
Let $\sV$ be a pseudoeffective reflexive sheaf on $X$ such that 
$\sV \otimes \sO_F \simeq \sO_F^{\oplus \rk \sV}$   for a general fibre $F$.
Then one of the following holds
\begin{itemize}
\item We have $\kappa(X, \sV) \geq 0$
\item A tautological class of $\sV$ is nef in codimension one.
\end{itemize}
\end{proposition}

\begin{remark} \label{remark-meaning}
Let us start by clarifying the statement in the case where $\sV$ is not necessarily locally free: let $\holom{p}{\PP'(\sV)}{X}$ and
$$
\holom{r \circ \nu}{P}{\PP'(\sV)}
$$ 
be a birational morphism from a projective manifold $P$ as in Definition \ref{definitiontautological}
and let 
$$
\zeta := c_1(\sO_{P}(1) \otimes \sO_P(\Lambda))  \in N^1(P)
$$
be a tautological class of $\sV$. 
Set $\pi := p \circ r \circ \nu$. By the construction of $\Lambda$ 
\cite[III,5.10(2)]{Nak04}
the restriction of $\zeta$ to a divisor $Z$ that is $\pi$-exceptional is not pseudoeffective. As we have seen in the discussion after Definition \ref{definitiontautological} these exceptional divisors do not play any role for the pseudoeffectivity of the reflexive sheaf.
 We therefore say that $\zeta$ is nef in codimension one if for 
every prime divisor $Z \subset P$ that is not  $\pi$-exceptional, the restriction
$\zeta|_Z$ is pseudoeffective. Note that this is slightly weaker than $\zeta$ being modified nef in the sense of Boucksom \cite[Prop.2.4]{Bou04}.
\end{remark}

The proof of Proposition \ref{proposition-dichotomy} uses at a key point
the following basic fact:

\begin{lemma} \label{lemmaeffirreduciblefibres}
Let $X$ be a projective manifold of dimension $n$ that admits a fibration $\holom{f}{X}{C}$ onto a curve $C$ such that all the fibres are irreducible.

Let $D$ be a divisor class on $X$ such that 
$\sO_F(D) \simeq \sO_F$. Then we have
$D \sim_\Q \lambda F$ for some $\lambda \in \Q$.
\end{lemma}

\begin{proof}
Since $\sO_F(D) \simeq \sO_F$ the direct image sheaf $f_* \sO_X(D)$ is torsion-free of rank one, so locally free. Moreover the evaluation morphism
$$
f^* f_* \sO_X(D) \rightarrow \sO_X(D)
$$
is surjective on the general fibre, so we can write
$D= f^* f_* \sO_X(D) + E$, where $E$ is an effective divisor supported on fibres of $f$. 
Since $f$ has irreducible fibres we have $E \sim_\Q f^* M$ for some $\Q$-divisor class $M$ on $C$. 
\end{proof}

Note that Lemma \ref{lemmaeffirreduciblefibres} immediately implies Proposition \ref{proposition-dichotomy} in the case $\rk \sV=1$, so in the next proof we can suppose $\rk \sV>1$:

\begin{proof}[Proof of Proposition \ref{proposition-dichotomy}] 
We use the notation introduced in Remark \ref{remark-meaning}.
We argue by contradiction and assume that 
$\zeta$ is not nef in codimension one. Thus there exists 
a prime divisor $Z \subset P$ such that the restriction $\zeta|_Z$ is not pseudoeffective. Moreover $Z$ is not $\pi$-exceptional
(cf. Remark \ref{remark-meaning}), so $\pi(Z)$ has codimension at most one.

{\em 1st case. Assume that $\pi(Z)$ is contained in an $f$-fibre.}
Since we assume that $f$ has irreducible fibres, this implies that set-theoretically $\pi(Z)=\fibre{f}{y}$ for some $y \in Y$ and hence
$$
(f \circ \pi)^* y = Z + E
$$
where $E$ is a union of $\pi$-exceptional divisors. The restriction of $\zeta$
to any irreducible component of $E$ is not pseudoeffective (cf. again
Remark \ref{remark-meaning}), so the restriction of  $\zeta$ to every irreducible component of $
(f \circ \pi)^* y$ is not pseudoeffective. Yet by assumption $\sV \otimes \sO_F \simeq
\sO_F^{\oplus \rk \sV}$, so the restriction of $\zeta$ to a general $f \circ \pi$-fibre
is effective. Since $f \circ \pi$ is flat and $\sO_P(\zeta)$ is locally free we obtain by semicontinuity that $\zeta$ has a non-zero section on $(f \circ \pi)^* y$. In particular the restriction to at least one irreducible component is pseudoeffective, a contradiction.

{\em 2nd case. Assume that $\pi(Z)$ is not contained in an $f$-fibre.}
Let
\begin{equation} \label{help2}
\zeta = P + (\mu Z + E)
\end{equation}
be the divisorial Zariski decomposition of $\zeta$ where $P$ is the positive part. Since $\zeta|_Z$ is not pseudoeffective, we have $\mu>0$.

Let $F$ be a general $f$-fibre. Since  $\sV \otimes \sO_F \simeq \sO_F^{\oplus \rk \sV}$ we have
$$
\PP(\sV \otimes \sO_F) \simeq F \times \PP^{\rk \sV-1}
$$ 
and the restriction of $\zeta$ to  $\PP(\sV \otimes \sO_F)$ is 
isomorphic to $p_{\PP^{\rk \sV-1}}^* H$ where $H$ is a hyperplane
on $\PP^{\rk \sV-1}$. In particular
$\sO_{F \times t}(\zeta)$ is trivial for for every $t \in \PP^{\rk \sV-1}$.

Since $\pi(Z)$ is not contained in an $f$-fibre, the intersection of
$Z$ with $\PP(\sV \otimes \sO_F)$ is not empty. Thus restricting \eqref{help2} 
to $\PP(\sV \otimes \sO_F)$, we obtain that
\begin{equation} \label{help3}
Z \cap \PP(\sV \otimes \sO_F) \sim p_{\PP^{\rk \sV-1}}^* (\lambda H)
\end{equation} 
for some $\lambda>0$. In particular this implies that $\pi(Z)=X$. 

We now claim that we can write
$$
[Z] \sim \lambda \zeta + \pi^* M + E,
$$
where $M$ is a Cartier divisor class on $X$, and $E$ is a not necessarily effective $\pi$-exceptional divisor.
In fact the reflexive sheaf $\sV \rightarrow X$ is locally free in codimension two, 
so if $B \subset X$ is the locus where $\sV$ is not locally free, the 
preimage $\fibre{\pi}{B}$ consists of $\pi$-exceptional divisors.
Since 
$$
\mbox{Pic}(\PP(\sV \otimes \sO_{X \setminus B})) \simeq \pi^* \mbox{Pic}(X \setminus B)
\oplus \Z \zeta \simeq
\pi^* \mbox{Pic}(X)
\oplus \Z \zeta,
$$
the claim follows from the usual exact sequence \cite[II,Prop.6.5]{Har77}.

Restricting over a general fibre $F$, we obtain from \eqref{help3} that
$$
p_{\PP^{\rk \sV-1}}^* \sO_{\PP^{\rk \sV-1}}(\lambda)
\simeq
\sO_{\PP(\sV \otimes \sO_F)}(Z)
\simeq
(p_{\PP^{\rk \sV-1}}^* \sO_{\PP^{\rk \sV-1}}(1))^{\otimes \lambda}
\otimes \pi^* \sO_F(M),
$$
so $\sO_F(M) \simeq \sO_F$. Since $f$ has irreducible fibres, we know by Lemma \ref{lemmaeffirreduciblefibres} that $M \sim_\Q f^* M_Y$ with $M_Y$ a $\Q$-divisor class on $\PP^1$. In conclusion we obtain that $Z$ defines a section of
$$
S^{[\lambda]} \sV \otimes \sO_X(f^* M_Y)
$$
over $X \setminus B$. Since $B$ has codimension at least two and $S^{[\lambda]} \sV$ is reflexive, this section extends to $X$. 
Thus by construction of the tautological class, cf. \eqref{push-tautological}, we obtain a global section of $\sO_P(\lambda \zeta+(f \circ \pi)^* M_Y)$ which vanishes on $Z+E'$ with $E'$ an {\em effective} $\pi$-exceptional divisor.

If $\deg M_Y  \leq 0$, we are already done: for some $d \in \N$ the $\Q$-divisor class $d M_Y$ is integral and anti-effective, so 
$$
d \lambda \zeta \sim_\Q d \left(Z+E'+(f \circ \pi)^* M_Y^* \right)
$$
is effective.

If $\deg M_Y>0$ we know by Proposition \ref{proposition-base-positivity} that $\kappa(X, S^{[\lambda]} \sV \otimes \sO_X(f^* M_Y)) \geq \rk \sV-1 \geq 1$.
If two sections of a symmetric power of $S^{[\lambda]} \sV \otimes \sO_X(f^* M_Y)$
coincide over $X \setminus B$, they coincide over $X$. Since $Z$ is the unique
irreducible component of 
$$
\lambda \zeta+(f \circ \pi)^* M_Y  \sim_\Q Z+E
$$
that is not $\pi$-exceptional, we obtain that $\kappa(P, Z) \geq 1$. 
Yet $Z$ is in the negative part of the divisorial Zariski decomposition of $\zeta$,
so $\kappa(P, Z)=0$ \cite[Prop.3.13]{Bou04}. This gives the final contradiction.
\end{proof}

\end{document}